\documentclass[preprint]{imsart}
\RequirePackage[OT1]{fontenc}
\RequirePackage{amsthm,amsmath}
\RequirePackage[sort,numbers]{natbib}
\RequirePackage[citecolor=red,urlcolor=blue]{hyperref}

\usepackage{amsfonts,amssymb}

\startlocaldefs

\theoremstyle{plain}
\newtheorem{lemma}{Lemma}[section]
\newtheorem{proposition}{Proposition}[section]
\newtheorem{theorem}{Theorem}[section]

\DeclareMathOperator*{\argmin}{arg\,min}

\let\oldenumerate\enumerate
\renewcommand{\enumerate}{
	\oldenumerate
	\setlength{\itemsep}{1pt}
	\setlength{\parskip}{0pt}
	\setlength{\parsep}{0pt}
}

\renewcommand\labelenumi{(\theenumi)}

\endlocaldefs

\begin{document}

\begin{frontmatter}

\title{On the robustness of minimum norm interpolators and regularized empirical risk minimizers}

\runtitle{Robust interpolation and regularization}


\begin{aug}

\author{\fnms{Geoffrey} \snm{Chinot}\ead[label=e1]{ geoffrey.chinot@stat.math.ethz.ch}},
\author{\fnms{Matthias} \snm{L\"offler}\ead[label=e2]{matthias.loeffler@stat.math.ethz.ch}}
\and 
\author{\fnms{Sara} \snm{van de Geer}\ead[label=e3]{sara.vandegeer@stat.math.ethz.ch}}

\address[]{Seminar for Statistics, Department of Mathematics, ETH Z\"urich, Switzerland \\
Emails:~\printead*{e1},\\
\printead*{e2}, \printead*{e3}
}

\end{aug}
\begin{abstract}
This article develops a general theory for minimum norm interpolating estimators and regularized empirical risk minimizers (RERM) in linear models in the presence of additive, potentially adversarial, errors.  In particular, no conditions on the errors are imposed. A quantitative bound for the prediction error is given, relating it to the Rademacher complexity of the covariates, the norm of the minimum norm interpolator of the errors and the size of the subdifferential around the true parameter.

The general theory is illustrated for Gaussian features and several norms: The $\ell_1$, $\ell_2$, group Lasso and nuclear norms.  In case of sparsity or low-rank inducing norms,  minimum norm interpolators and RERM yield a prediction error of the order of the average noise level, provided that the overparameterization is at least a logarithmic factor larger than the number of samples and that, in case of RERM, the regularization parameter is small enough.

Lower bounds that show near optimality of the results  complement the analysis.

\end{abstract}




\end{frontmatter}

\section{Introduction}

Finding solutions to noisy, underdetermined systems of linear equations is a fundamental problem in mathematics and other areas of science. Applications of this problem range from magnetic resonance imaging 
to inferring movie ratings, and gives rise to a wide range of statistical models such as sparse linear regression or matrix completion. Considering a Hilbert space $\mathbb{H}$, equipped with inner product $\langle \cdot, \cdot \rangle$, these models can  be expressed as observing data pairs $(Y_i, X_i), ~Y_i \in \mathbb{R}, ~X_i~ \in \mathbb{H}, ~i=1, \dots, n$, such that 
\begin{align} \label{intro lin mod}
    Y_i = \langle X_i, h^* \rangle + \xi_i.
\end{align}
The $\xi_i$'s denote unspecified error terms. In the last 50 years a thorough, general understanding of these problems has developed. By now, it is common knowledge in statistics and applied mathematics that {\em regularization} is able to significantly boost performance, particularly in presence of low-dimensional intrinsic structure \cite{CandesRombergTao06,RechtFazelParrilo10,ChandrasekaranRechtParriloWillsky12,CaiLiangRakhlin16,lecue2018regularization,GaovdVZhou20} but also in other settings via shrinkage \cite{HoerlKennard70, Casella80,BenningBurger18}.

By contrast, in the  machine learning literature a completely different paradigm has developed. Explicit regularization does  often not play a role. Instead, algorithms such as neural networks, AdaBoost or random forests  are typically overparameterized and run until the training error equals zero, i.e. they interpolate the data. Nevertheless, their empirical test error often improves upon algorithms that would be preferable from a theoretical  point of view \cite{ZhangBengioHardtRechtVinyals17,WynerOlsonBleichMease17}. 
Numerical experiments in \cite{BelkinHsuMaMandal19} further validated these findings. They showed that for overparameterized algorithms the test error decreases monotonically in the number of parameters. 

The heuristic explanation for this decrease is that overparameterized algorithms such as neural networks trained with gradient descent \cite{SoudryHofferNacsonGuasekarSrebro18} or AdaBoost \cite{RossetZhuHastie04,Telgarsky13, LiangSur20,ChinotKuchelmeisterLoefflervdG21} converge to minimum norm interpolators. It is then argued that this leads to implicit regularization and hence a good test error. 

A mathematical understanding of this novel heuristic has only began to form recently. Most theoretical work has investigated 
minimum Euclidean norm interpolators in the linear model \eqref{intro lin mod}, see for example \cite{BartlettLongLugosiTsigler20,HastieMontanariRossetTibshirani19,ChinotLerasle20}. In particular, it has been shown that if the eigenvalues of the covariance matrix of the features  ${X}_{i}$ fulfill certain decay conditions and the norm of $h^*$ is small enough, then the minimum $\ell_2$-norm interpolator has prediction error tending to zero. Other models where minimum Euclidean norm type  interpolators were investigated are kernel interpolation \cite{LiangRakhlin20}, classification \cite{MontanariRuanSohnYan20,MuthukumarNarangSubramanianBelkinHsuSahai20,ChatterjiLong21,LiangRecht21,ChinotKuchelmeisterLoefflervdG21} and  two-layer random features regression \cite{MeiMontanari19}. 

However, as the Euclidean norm is rotation invariant, minimum $\ell_2$-norm interpolators cannot adapt to underlying sparsity or other low-dimensional intrinsic structure.  Another issue is that they induce shrinkage to zero, leading to, in general non-negligible, bias. In addition, in the overparameterized regime the popular AdaBoost algorithm is closely connected to the $\ell_1$-norm \cite{RossetZhuHastie04,Telgarsky13}. This motivates to study minimum norm interpolators for norms that induce  low-dimensional structure such as sparsity. Pioneering work \cite{Wojtaszczyk10} has shown a stability result for the minimum $\ell_1$-norm interpolator, basis pursuit \cite{ChenDonohoSaunders98}, in sparse linear regression. In particular,  basis pursuit attains Euclidean error of the size of the average noise level, when the covariates $X_i$ are standard Gaussian, the number of parameters is larger than the number of samples by a logarithmic factor and the sparsity of $h^*$ is small enough compared to the number of samples \cite{Wojtaszczyk10}. This approach has been further extended to sub-exponential \cite{Foucart14} and heavy tailed \cite{KrahmerKuemmerleRauhut18} features $X_i$ and quadratic \cite{KrahmerKuemmerleMelnyk20} measurements.  Independently, a weaker version of the result by \cite{Wojtaszczyk10} was recently proven by \cite{JuLinLiu20}. 

In this article, we present a unified perspective on minimum norm interpolators and regularized empirical risk minimizers (RERM) for regression type problems, building on the general theory for regularized empirical risk minimizers developed in \cite{lecue2017regularization,lecue2018regularization}. Contrary to \cite{lecue2017regularization,lecue2018regularization}, our results for RERM hold even when the regularization parameter is small. These are, as far as we know, the first general results on regularized empirical risk minimizer in such a setting. Our theory encompasses both Euclidean norm regularization as well as sparsity or low-rank inducing norms. We relate the prediction error to three quantities: the Rademacher complexity of the covariates, the size of the subdifferential of the regularization norm $\|\cdot\|$ at $h^*$ and the norm of the minimum norm interpolated noise. We also present a minimax lower bound for adversarial errors which is of order of the average noise level and matches our upper bound for most of our examples.  Examples of our theory include sparse linear regression with $\ell_1$ or group Lasso penalty, minimum $\ell_2$-norm interpolation in linear regression and nuclear norm minimization in trace regression. \\

We present the framework and main results in Section \ref{sec main main}. We start with our main theorem, Theorem \ref{Main general theorem},  on minimum norm interpolators  in Section \ref{subsec gen theory}.  Next, in Section \ref{subsec gen theory RERM}, we present Theorem \ref{thm RERM}, that generalizes our results from Section \ref{subsec gen theory} under one additional assumption to regularized empirical risk minimizers (RERM). We finish off our general results with a minimax-lower bound in Section \ref{subsec lower}. 
Afterwards, we present examples with Gaussian features to illustrate the presented theory, starting with minimum $\ell_1$-norm interpolation in sparse linear regression in Section \ref{subsec sparse regression}. 
We finish with concluding remarks in Section \ref{subsec concl} and, finally, present all proofs in Section \ref{sec proofs} and Appendices \ref{lemma supp bp} to \ref{appendix proof lb}.

\subsection*{Notation}
We consider throughout a Hilbert space $\mathbb{H}$, equipped with Euclidean norm  $\|\cdot \|_2$ and induced by the inner product $\langle \cdot , \cdot \rangle$.  If $\mathbb{H}=\mathbb{R}^p$ we use the canonical basis and denote the $\ell_1$-norm and $\ell_\infty$-norm by $\| \cdot \|_1$ and $\| \cdot \|_\infty$, respectively. Moreover, in this case, we denote the number of non-zero elements of $h$ by $\|h\|_0$.  We also use the matrix norms $\| \cdot \|_{\textnormal{op}}$ and $\| \cdot \|_{S_1}$, denoting the operator norm and nuclear (Schatten-1) norm, respectively. For a norm $\|\cdot\|$ acting on $\mathbb{H}$ we denote the dual norm by $\| \cdot \|^* := \sup_{ \|x \| \leq 1} \langle \cdot, x \rangle.$

For two sequences $a_n$, $b_n$, if there  exists a constant $c>0$ such that $a_n \leq cb_n$ for all $n$, we write $a_n \lesssim b_n$.  Likewise $a_n \asymp b_n$ if $a_n \lesssim b_n$ and $b_n \lesssim a_n$. We define $a \vee b:=\max(a,b)$. By $[p]$ we denote the enumeration $\{1, \dots, p\}$ and by $\mathbf{1}(\cdot )$ the indicator function.   

\section{General framework and theory} \label{sec main main}
\subsection{Model assumptions and estimators} \label{sec main}
We consider the following linear model where observations $(Y_i, X_i)_{i=1}^n$ are generated as 
\begin{align*}
    Y_ i= \langle X_i , h^* \rangle  +\xi_i, \quad i=1,\cdots,n.
\end{align*}
Here $h^* \in \mathbb{H}$ denotes the unknown parameter to be estimated, the $X_i$'s denote random covariates with $X_i \overset{i.i.d.}{\sim} \mu, ~i=1,\dots,n,$  and $\xi = (\xi_1,\cdots, \xi_n)$ is a, possibly deterministic or adversarial, error vector. We call $\|\xi\|_2/\sqrt{n}$ the average noise level. For a norm $\|\cdot\|$ defined on $\mathbb H$, the minimum norm interpolating estimator is given by
\begin{align} \label{est gen form}
    \hat h \in \argmin_{h \in \mathbb{H}} \|h\| ~~~~\text{subject to}~~~ Y_ i= \langle X_i , h \rangle, \quad i=1,\dots,n. 
\end{align}
A solution to \eqref{est gen form} exists under weak assumptions, for instance when $\dim (\mathbb{H}) \geq n$ and the $X_i$ are linearly independent, which occurs with probability one if $\mu$ is continuous. 
Similarly, we define the regularized empirical risk minimizer (RERM) with regularization parameter $\lambda>0$ by
\begin{equation} \label{def RERM}
    \hat h_{\lambda} \in \argmin_{h \in \mathbb H} \frac{1}{n}\sum_{i=1}^n \left( Y_i - \langle X_i,h \rangle \right)^2 + 2 \lambda \| h\| . 
\end{equation}

We consider the prediction errors
$$\| \langle  h - h^*, X \rangle \|_{L_2(\mu)}^2 := \int \langle h - h^*, x \rangle^2 \textnormal{d}\mu(x),\enspace  h \in \{ \hat h, \hat h_{\lambda} \}, $$
where $X$ denotes an independent copy of $X_1$.


\subsection{Error bounds for minimum norm interpolators} \label{subsec gen theory}

In the following we give a general bound on the prediction error of minimum norm interpolators. 
Our main result, Theorem \ref{Main general theorem}, relies on three quantities and one assumption: the local Rademacher complexity \cite{lecue2018regularization}, the small-ball assumption \cite{mendelson2014learning,koltchinskii2015bounding}, the norm of the minimum norm interpolator of the errors and, optionally,  the size of the subdifferential of $\|\cdot\|$ at $h^*$. We now present and discuss those quantities. \\

\noindent \textbf{Local Rademacher Complexity.} In the same spirit as~\cite{lecue2018regularization}, we define for $\gamma >0$, $r^*(\gamma)$ as the solution of the following fixed point equation: 
\begin{equation} \label{rademacher_com}
    r^*(\gamma): =  \inf \big \{ r >0 : \mathbb E \sup_{h \in B : \| \langle X,h \rangle \|_{L_2(\mu)} \leq r } \sum_{i=1}^n \varepsilon_i \langle X_i, h \rangle \leq \gamma n r \big \} ,
\end{equation}
where $\varepsilon_1,\cdots, \varepsilon_n$ denote i.i.d. Rademacher random variables independent of the $X_i$'s, $B$ denotes the unit $\|\cdot\|$-ball, $B := \{ h \in \mathbb H: \|h\| \leq 1 \}$, and the expectation is taken over the $\varepsilon_i$'s and $X_i$'s. The complexity parameter $r^*(\gamma)$ measures the complexity of the class $\{ \langle \cdot, h \rangle, ~~ h \in \mathbb H \}$ locally around zero. The localization is made with respect to the $L_2 (\mu)$ norm and the interpolation norm $\| \cdot\|$. As the overparameterization increases, $r^*(\gamma)$ increases too. When considering regularized empirical risk minimizers such as the Lasso, \cite{lecue2018regularization} have recently shown that this is the key quantity for determining convergence rates in the noiseless regime. \\

\noindent \textbf{Small-ball assumption (SB).}
The only assumption we make on the distribution of the $X_i$'s is the small-ball assumption~\cite{mendelson2014learning,koltchinskii2015bounding}. We  assume that there exist $\kappa, \delta$ such that for every $h \in \mathbb H$ and for $X$ denoting an independent copy of $X_1$
\begin{equation} \label{small_ball}
    \mathbb P \left( | \langle X, h \rangle | \geq \kappa   \| \langle h , X \rangle \|_{L_2(\mu)} \right) \geq \delta .
\end{equation}
The small-ball assumption allows for heavy tailed covariate distributions and is implied by a $L_4$-$L_2$-moment condition, where  $\| \langle X, h \rangle \|_{L_4(\mu)} := \big( \int \langle h , x \rangle^4 \textnormal{d}\mu(x) \big)^{1/4}$.
\begin{lemma} \label{paley argument}
Assume that there exists $B >0$ such that for every $h \in \mathbb H$, $\| \langle h , X \rangle \|_{L_4(\mu)}  \leq B \| \langle h , X \rangle \|_{L_2(\mu)} $, then~\eqref{small_ball} holds for $\kappa \in [0,1]$ and $\delta = (1-\kappa^2)/B^4$.
\end{lemma}
\begin{proof}
The proof follows by applying  the Paley-Zygmund inequality (see e.g. \cite{de2012decoupling}).
\end{proof}
For example, the small-ball assumption is fulfilled by the  Gaussian distribution, but also the Student-$t$-distribution.  \\

\noindent \textbf{Interpolated noise.}
We define the minimum norm interpolator of the noise as 
\begin{equation} \label{interpole noise}
\hat \nu \in \argmin_{h \in \mathbb{H}} \|h\| ~~~~\text{subject to}~~~ \xi_i= \langle X_i , h \rangle, \quad i=1,\cdots,n.
\end{equation}
Our bound on the prediction error bound will depend on $\|\hat \nu\|$. Typically, when the  overparameterization grows this quantity becomes smaller, cancelling the increase of $r^*(\gamma)$. In our examples we are able to bound $\|\hat \nu\|$ when the features $X_i$ are Gaussian distributed and the overparameterization exceeds $n$ by a logarithmic factor. We discuss extensions below in Section \ref{subsec concl}.\\ 

\noindent \textbf{Subdifferential.}
The last tool we need enables to measure the size of the subdifferential of the regularization norm $\| \cdot \|$ around the target $h^*$. If $h^*$ has low-dimensional intrinsic structure and $\|\cdot\|$ promotes this structure, then using the subdifferential of $\|\cdot\|$ allows to obtain error rates that only depend on $\xi$. Similar as  \cite{lecue2018regularization}, we define
\begin{align} \label{def subdiff cond}
\Delta(\gamma, h^* ) = \inf_{h \in H_{r,\gamma}} \sup_{g \in \partial(\| \cdot \|)_{h^*} } \langle g, h \rangle, \end{align}
where $\partial(\| \cdot \|)_{h^*} $ denotes the subdifferential of $\| \cdot \|$ evaluated at $h^*$ defined as
$$
\partial(\| \cdot \|)_{h^*}  = \{ g \in S^*: \langle g,h \rangle = \| h^* \| \},
$$
and where $S^*$ denotes the unit sphere for the dual norm $\|\cdot \|^*$, and 
\begin{align*}
H_{r,\gamma} := \{ h \in \mathbb H: \|\langle X, h \rangle \|_{L_2(\mu)} \leq r^*(\gamma) ~\textnormal{  and  }\| h \| = 1 \}.
\end{align*}
In order to obtain an improved error bound, we optionally assume that $\Delta(\gamma, h^*) \geq \zeta$ for some constant $\zeta >0$.

The main intuition behind this condition comes from the fact that low-dimensional structure inducing norms are typically non-differentiable and thus can have locally large subdifferentials. %
On the other hand, when $\| \cdot \|$ is smooth, the subdifferential reduces to the gradient and there is, in general, no hope of obtaining bounds independent of $\| h^*\|$. 
Hence, $\Delta(\gamma, h^* )$ can expected to be large when the interpolation norm induces low-dimensional structure and $h^*$ lies in a subspace of $\mathbb H$ with small complexity compared to $n$. 

We are now in position to state the main result of this section.

\begin{theorem} \label{Main general theorem}

Assume that the small-ball assumption (SB) holds with $\kappa,\delta >0$ and set $\gamma = \kappa \delta / 32$. Then, with probability at least $1 - 2\exp(-\delta^2 n /16)$
\begin{align} \label{thm main bound 1}
\| \langle X, \hat h - h^* \rangle \|_{L_2(\mu)} \leq \frac{ \sqrt 8}{\kappa \sqrt{\delta}} \frac{\| \xi \|_2}{  \sqrt n} \vee r^*(\gamma) \big( 2 \| h^* \| + \| \hat \nu \| \big).
\end{align}
Moreover, if there exists $\zeta > 0$ such that $\Delta(\gamma, h^* ) \geq \zeta$, then on the same event we have that 
\begin{align} \label{thm main bound 2}
\| \langle X, \hat h - h^* \rangle \|_{L_2(\mu)} \leq \frac{\sqrt 8}{ \kappa \sqrt{\delta}} \frac{\| \xi \|_2}{ \sqrt n} \vee  \frac{r^*(\gamma) \| \hat \nu \|}{\zeta}.
\end{align}
\end{theorem}

Theorem~\ref{Main general theorem} holds under weak assumptions on the data $(X_i,Y_i)_{i=1}^n$ and without any assumption on the error vector $\xi$.
In particular, $\xi$ may be adversarial or deterministic.

When the problem is noise free, i.e. $\xi=0$, then the first term in the error bounds equals zero and, likewise, $\| \hat \nu \| = 0$. Hence, in this case the error bound \eqref{thm main bound 1} is of order of $\| h^* \| r^*(\gamma)$ which is the prediction error obtained in Theorem 1.9 in~\cite{lecue2017regularization} for the regularized empirical risk minimizer (RERM) with optimal tuning parameter choice.  If a lower bound on $\Delta(\gamma, h^*)$ is available such that \eqref{thm main bound 2} holds, we improve upon this bound and achieve exact recovery of $h^*$ in the absence of noise. 

In the presence of noise, the error bounds \eqref{thm main bound 1} and \eqref{thm main bound 2} depend on the average noise level, the local Rademacher complexity multiplied with the norm of the minimum noise interpolator and, if no subdifferential information is available, the norm of $h^*$. 
As such, they are weaker than the bounds that can be obtained for RERM \cite{lecue2017regularization, lecue2018regularization} when additional assumptions on the errors $\xi$ are made. However, the bounds in \cite{lecue2017regularization, lecue2018regularization} rely crucially on the assumption that the $\xi_i$ are random, have mean zero, more than two bounded moments and are independent of the $X_i$  and $h^*$. By contrast, Theorem \ref{Main general theorem} allows for errors that may adversarially rely on the $X_i$ or $h^*$ or may be deterministic.  

The quantity$\|\xi\|_2/\sqrt{n}$ appears naturally in the proof of Theorem \eqref{Main general theorem} due to the interpolating nature of $\hat h$ which leads to the equality $\sum_{i=1}^n \langle X_i, \hat h-h\rangle^2 = \|\xi\|_2^2$. 



In the examples below, we show that $r^*(\gamma) \| \hat \nu \|$ is of order of the average noise level $\| \xi \|_2 / \sqrt n$. To do so, we use a general approach to control $\|  \hat \nu\|$ based on the dual formulation of $\| \cdot \|$. 
\begin{lemma}\label{duality formulation}
Let $\| \cdot \|^*$ be the dual norm of $\| \cdot \|$. 
Suppose that $\dim(\mathbb{H}) \geq n$ and that the $X_i$ are linearly independent. 
Then, for $\mathcal{S}^{n-1}$ denoting the unit sphere with respect to $\| \cdot\|_2$, we have that 
\begin{equation} \label{bound lemma dual}
   \frac{\|\xi\|_2}{\left \| \sum_{i=1}^n \frac{\xi_i}{\|\xi\|_2} X_i \right  \|^*} \leq  \| \hat \nu \| \leq \frac{\| \xi \|_2 }{  \inf_{v \in \mathcal S^{n-1}} \| \sum_{i=1}^n v_i X_i \|^*   }.
\end{equation}
\end{lemma}
In our examples we show how to lower bound the infimum on the right-hand-side in \eqref{bound lemma dual} when the distribution of the $X_i$ is Gaussian and the number of free parameters exceeds $n$ by a logarithmic factor. Moreover, when additionally the $\xi_i$'s are independent of the $X_i$'s, the resulting upper bound is in all our examples with Gaussian features of the order of the lower bound on $\|\hat \nu\|$ in \eqref{bound lemma dual}. However, for features that are not Gaussian the above bound is, in general, not tight and we discuss this further in Section \ref{subsec sparse regression} in context of the sparse linear model. 

We conclude this section with a general routine for the application of Theorem ~\ref{Main general theorem}. In Section \ref{sec examples} we obtain bounds for several examples by using this scheme.  \\

\fbox{\parbox{\textwidth}{
\begin{itemize}
    \item[] \textbf{Step 1:} Define a norm $\|\cdot\|$ on $\mathbb H$.\\
    \item[] \textbf{Step 2:} Verify the small-ball condition with parameters $\kappa,\delta >0$.\\
    \item[] \textbf{Step 3:} Compute the local Rademacher complexity and deduce a bound for  $r^*(\gamma)$.\\
    \item[] \textbf{Step 4:} Use Lemma~\ref{duality formulation} to control $\|\hat \nu \|$ with high probability. \\
    \item[] \textbf{Step 5:} If the subdifferential of $\|\cdot\|$ at $h^*$ is large enough such that $\Delta(\gamma,h^*)$ can be lower bounded, use the second part of Theorem~\ref{Main general theorem}, otherwise, use the first part. 
\end{itemize}}}

\subsection{A generalization to regularized empirical risk minimizers (RERM)} \label{subsec gen theory RERM}
In this section we generalize the results obtained in Section~\ref{sec main} to the RERM, which we recall is defined for $\lambda >0$ by 
\begin{equation}
    \hat h_{\lambda} \in \argmin_{h \in \mathbb H} \frac{1}{n}\sum_{i=1}^n \left( Y_i - \langle X_i,h \rangle \right)^2 + 2 \lambda \| h\| . 
\end{equation}
RERM can be seen as a generalization of minimum norm interpolators in overparameterized regimes, as for $\lambda \rightarrow 0^+$ RERM converges to the corresponding minimum norm interpolator. 

To control the error-rate $\| \langle \hat h_{\lambda} - h^*, X \rangle \|_{L_2(\mu)}$ we need a variant of the subdifferential condition defined in~\eqref{def subdiff cond}. We define
\begin{equation} \label{variant SE}
    \bar \Delta(\gamma,h^*) = \sup_{h \in \bar H_{r,\gamma}} \inf_{g \in \partial \left( \| \cdot \\| \right)_{h^*}} \langle g,h \rangle ,
\end{equation}
where 
$$
\bar H_{r,\gamma} = \{ h \in \mathbb H: \| \langle X,h \rangle \|_{L_2(\mu)} = r^*(\gamma) \text{ and } \| h \| \leq 1 \} . 
$$
Our next result will require an upper bound on $\bar \Delta(\gamma,h^*)$ in addition to the previously assumed lower bound on $\Delta(\gamma, h^*)$. Contrary to the lower bound on  $\Delta(\gamma, h^*)$, the upper bound on $\bar \Delta(\gamma,h^*)$ depends on the complexity of the subspace of $\mathbb H$ where $h^*$ lies in. In the same spirit as the subdifferential condition defined in~\eqref{def subdiff cond}, to upper bound $ \bar \Delta(\gamma,h^*)$, we need the subdifferential of $\| \cdot \|$ to be large at $h^*$ to be able to find an appropriate $g \in \partial (\|\cdot\|)_{h^*}$ that leads to a sufficiently small upper bound. \\
We are now in position to state the main result of this section. 
\begin{theorem} \label{thm RERM}
Assume that the small-ball assumption (SB) holds with $\kappa,\delta >0$ and set $\gamma = \kappa \delta/32$. Then with probability at least $2\exp(-\delta^2n/16)$, the estimator $\hat h_{\lambda}$, $\lambda >0$,  satisfies
\begin{align*} 
    \| \langle \hat h_{\lambda} - h^* , X \rangle  \|_{L_2(\mu)} \leq &  \frac{2\sqrt 8}{\sqrt \delta \kappa} \left( \frac{\| \xi \|_2}{\sqrt n} + \sqrt{\lambda \|h^* \|}\right)  \vee r^*(\gamma) \left(  2\|h^* \| + \|\hat \nu\| \right) . 
\end{align*}
Moreover, if there exist $\zeta, \bar \zeta >0$ such that $\Delta(\gamma,h^*) \geq \zeta$ and $\bar \Delta(\gamma,h^*) \leq \bar \zeta$, then on the same event we have that
$$
\| \langle \hat h_{\lambda} - h^* , X \rangle \|_{L_2(\mu)}  \leq \frac{4\sqrt 8}{\sqrt \delta \kappa}\frac{\| \xi \|_2}{ \sqrt n} + \frac{32}{\delta \kappa^2}\frac{\lambda  \bar \zeta}{r^*(\gamma)} \vee \frac{r^*(\gamma)\|\hat \nu\|}{\zeta }  .  
$$
\end{theorem} 
Theorem~\ref{thm RERM} holds under weak assumptions on the data $(X_i,Y_i)_{i=1}^n$ and, as before, without any assumption on the errors $\xi$.  When taking $\lambda \rightarrow 0^+$, we recover the rates obtained in Theorem~\ref{Main general theorem} for the minimum norm interpolator \eqref{est gen form} under slightly stronger assumptions due to the assumption needed on $\bar \Delta(\gamma,h^*)$. 

The proof of Theorem \ref{thm RERM} is slightly more involved than the proof of Theorem \ref{Main general theorem} as the cross term $\sum_{i=1}^n X_i(Y_i - \langle X_i, \hat h_\lambda\rangle)$ does not vanish anymore. This is due to the fact that $\hat h_\lambda$ is, in general, not interpolating the data. 

By contrast to \cite{lecue2017regularization,lecue2018regularization}, we are able to deal with the small $\lambda$ regime by relying on overparameterization that eventually enables control of $\|\hat \nu\|$ and by using the newly introduced subdifferential condition on $\bar \Delta(\gamma,h^*)$. 
This is, as far as we know, the first general result for the regularized empirical risk minimizer both for small values of the penalization parameter $\lambda$ and for adversarial noise. 

 \subsection{Lower bound} \label{subsec lower}
We now present a unified minimax lower bound for worst case type adversarial errors, that shows that the first term on the right hand sides in \eqref{thm main bound 1} and \eqref{thm main bound 2} in Theorem \ref{Main general theorem} is sharp in the presence of adversarial errors. It holds  for any subset $L \subset \mathbb{H}$ that contains a non-trivial linear subspace    and any  covariates $X_i$ that are i.i.d. and have bounded second moment.

\begin{proposition} \label{Prop lower bd gen} 
Suppose that for $0 < \epsilon < 1$ there exists a subset $L \subset \mathbb{H}$  such that $0 \in L$ and such that there exists $h_1 \in L$ satisfying $\|\langle h_1, X \rangle \|_{L_2(\mu)}^2=\epsilon^2/8$. Then, we have that 
\begin{align}
        \inf_{\tilde h} \sup_{h^* \in L, ~\xi: \|\xi\|_2^2\leq n \epsilon^2 } \mathbb{P}_{h^*, \xi} \bigg ( \|\langle X, \tilde h-h^*\rangle\|_{L_2(\mu)}^2 \geq \frac{ \epsilon^2}{16} \bigg  ) \geq \frac{3}{8}. 
\end{align}
\end{proposition}
The proof of Proposition \ref{Prop lower bd gen}  is a basic rewriting exercise and uses that, in the adversarial error case, $h^*$ is only identifiable in a $L_2(\mu)$-sense up to an error of order $\|\xi\|_2/\sqrt{n}$. In particular, we construct an adversarial error that depends on the $X_i$ such that it has the same distribution as the signal and then use Le-Cam's two point method to conclude. 

 When the subdifferential condition holds, then in all our examples $r^*(\gamma)\|\hat \nu\|/\zeta$ is of order of the average noise level $\|\xi\|_2/\sqrt{n}$, and hence in  those cases the bound of Theorem \ref{Main general theorem} is minimax optimal for adversarial errors.
Likewise, RERM achieves minimax optimal rates against adversarial errors when the regularization parameter $\lambda$ is small enough and the two subdifferential conditions are satisfied.

\section{Examples} \label{sec examples}

\subsection{Sparse linear model} \label{subsec sparse regression}
In case of the sparse linear model we consider $\mathbb{H}=\mathbb{R}^p$. We assume i.i.d. Gaussian design, i.e. $X_i \overset{i.i.d.}{\thicksim} \mathcal{N}(0,\Sigma)$, $i=1, \dots, n$, and that $h^*$ is $s$-sparse, $\|h^* \|_0 = s$. 

The canonical interpolating estimator in this setting is basis pursuit \cite{ChenDonohoSaunders98}:
\begin{align} 
    \hat h=\argmin_{h \in \mathbb{R}^p} \|h \|_1 \label{def l1 inter est}, ~~~~~~\textnormal{subject to}~~~\langle X_i, h \rangle = Y_i, ~~i=1, \dots, n.
\end{align}
Consequently, we choose $\|\cdot\|=\|\cdot\|_1$. Since $\|\cdot\|_1$ induces sparsity, the subdifferential condition of Theorem \ref{Main general theorem} is fulfilled provided that $s\log(p/n) \lesssim n$ and $\Sigma$ fulfills the restricted eigenvalue condition \cite{BickelRitovTsybakov09}. \\
\textbf{Restricted eigenvalue condition:} $\Sigma$ satisfies the restricted eigenvalue condition with parameter $\psi$ if 
$$ \|\Sigma^{1/2} h\|_2 \geq \psi \|P_I h\|_2,
$$
for all $ h \in \mathbb{R}^p$ satisfying $\|P_{I^c} h\|_1 \leq 3 \|P_I h\|_1$, where we denote by $I=\text{supp}(h^*)=\{i \in [p]:~h_i^* \neq 0\}$ and $(P_I h)_{i}:=h_i \mathbf{1}(i \in I), ~i \in [p]$. \\

We are now in position to state our main theorem for basis pursuit.
\begin{theorem} \label{thm cs}
Let $\beta\in (0,1)$. There exist constants $c_1,c_2>0 $  such that if $p \geq c_1n \log^{1/(1-\beta)}(n)$,  then, with probability at least $1-p\exp(-c_2n)$ 
\begin{align} \label{bound cs no sparsity}
    \|\Sigma^{1/2}(\hat h-h^*)\|_2^2 \lesssim  \bigg (1 \vee \sup_{\|b\|_1=1}\|\Sigma^{-1/2} b\|_1^2  \max_{i} \Sigma_{ii}\bigg )\bigg(   \frac{\|\xi\|_2^2}{\beta n}+ \|  h^* \|_1^2  \frac{\log(p/n)}{n}  \bigg). 
\end{align}
Moreover, suppose that $\Sigma$ satisfies the restricted eigenvalue condition with parameter $\psi$. If for some small enough constant $c_3 >0$ we have $\max_i \Sigma_{ii} s \log(p/n)/n\leq c_3 \psi^2$, then on the same event
\begin{align} \label{bound cs master}
    \|\Sigma^{1/2}(\hat h-h^*)\|_2^2 \lesssim \bigg (1 \vee \sup_{\|b\|_1=1}\|\Sigma^{-1/2} b\|_1^2 \max_{i} \Sigma_{ii}\bigg ) \frac{\|\xi\|_2^2}{\beta n} . 
\end{align}
\end{theorem}

Theorem \ref{thm cs} shows a phase transition in the behavior of $\hat h$. When $\max_i \Sigma_{ii}  s\log(p/s)/n  \lesssim \psi^2$, the prediction error is bounded by the average noise level in the overparameterized regime, matching up to constants the error bound for the regularized estimator with optimal tuning parameter choice in \cite{CandesRombergTao06}. As $\beta$ grows, i.e. the overparameterization becomes larger, this bound improves by a constant factor. Moreover, for $\beta=1/2$ and $\Sigma=I_p$ we recover the stability result of Theorem 1.1. in  \cite{Wojtaszczyk10}. When $\max_i \Sigma_{ii}  s\log(p/n)/n \gtrsim\psi^2$, an extra term depending on the $\ell_1$-norm of $h^*$ appears. A similar phenomenon has also been observed  in~\cite{lecue2017regularization}.

If the $X_i$'s consist of i.i.d. zero mean random variables with at least $\log(p)$ moments,  it is possible to obtain a similar result as in Theorem \ref{thm cs}. In this case the upper bound in Lemma \ref{duality formulation} is, in general, not sharp and it is necessary to bound $\|\hat \nu\|_1$ directly. This can be achieved by applying Theorem 5a in \cite{KrahmerKuemmerleRauhut18} to bound  $\|\hat \nu\|_1$ via its dual formulation.  In particular, in the bounds in Theorem \ref{thm cs} $\|\xi\|_2$ is then replaced by $\|\xi\|_2+\sqrt{\log(p/n)}\|\xi\|_\infty$.  If the subdifferential condition is fulfilled, this recovers the error guarantee for basis pursuit in Theorem 8a in \cite{KrahmerKuemmerleRauhut18} with a  different proof. By contrast, combining the upper bound in Lemma \ref{duality formulation} with Theorem 5a in \cite{KrahmerKuemmerleRauhut18} leads to an additional factor of $\sqrt{\log(p/n)}$ in front of $\|\xi\|_2$. Both bounds are, in general, tight as can be seen by considering i.i.d. Rademacher features.

In  view of Proposition \ref{Prop lower bd gen} with $L=\{ h \in \mathbb{R}^p: ~\|h\|_0 \leq s\}$, the error bound in \eqref{bound cs master} is for adversarial errors and well behaved covariance matrices with $\sup_{\|b\|_1=1} \|\Sigma^{-1/2} b\|_1^2 \max_i \Sigma_{ii} \lesssim 1$ minimax optimal for the class of $s$-sparse vectors.
This makes the heuristic reasoning in \cite{CandesRombergTao06} rigorous and shows that in absence of further {\em a priori} knowledge about the errors $\xi$ basis pursuit performs optimally and, in particular, as well as any regularized algorithm. 

Moreover, applying Theorem \ref{thm RERM}, we obtain similar results for the RERM, the Lasso \cite{Tibshirani96}, with an additional error term depending on the regularization parameter $\lambda$. 

\begin{theorem} \label{erm lasso}
Under the same conditions as in Theorem~\ref{thm cs} we have with probability at least $1-p\exp(-c_1n)$ for $\hat h_\lambda$, $\lambda >0$, that
\begin{align*} 
    \|\Sigma^{1/2}(\hat h_{\lambda}-h^*)\|_2^2  \lesssim &  \bigg (1 \vee \sup_{\|b\|_1=1}\|\Sigma^{-1/2} b\|_1^2  \max_{i} \Sigma_{ii}\bigg )\\
    & \cdot\bigg(    \frac{\|\xi\|_2^2}{\beta n} +\|  h^* \|_1^2  \frac{\log(p/n)}{n}+  \lambda \| h^* \|_1  \bigg). 
\end{align*}
Moreover, suppose $\Sigma$ satisfies the restricted eigenvalue condition with parameter $\psi$.
If for some small enough constant $c_2>0$ we have $\max_i \Sigma_{ii}  s \log(p/n)/n\leq c_2 \psi^2$, then on the same event
\begin{align*} 
    \|\Sigma^{1/2}(\hat h_{\lambda}-h^*)\|_2^2 \lesssim  \bigg (1 \vee \sup_{\|b\|_1=1}\|\Sigma^{-1/2} b\|_1^2 \max_{i} \Sigma_{ii}\bigg ) \left( \frac{\|\xi\|_2^2}{\beta n} + \frac{s\lambda^2 }{\psi^2} \right).
\end{align*}
\end{theorem}

It is again possible to extend Theorem \ref{erm lasso} to features that consist of i.i.d. zero mean random variables with at least $\log(p)$ moments by applying Theorem 5a in \cite{KrahmerKuemmerleRauhut18} to bound $\|\hat \nu\|_1$. As before, $\|\xi\|_2$ is then replaced by $\|\xi\|_2+\sqrt{\log(p/n)}\|\xi\|_\infty$.

\subsection{Group sparse linear model}
We now consider the group sparse linear model \cite{yuan2006model}. Here, the set of variables is partitioned into prescribed groups and we assume that only a few of those are relevant for the estimation process. We have again that $\mathbb{H}=\mathbb{R}^p$. 
For $S \subset [p]$, we denote by $h_S$ the projection of $h$ onto $\textnormal{span}(e_i, i \in S)$, where $(e_i)_{i=1}^p$ denotes the canonical basis of $\mathbb R^p$. We say that $h$ is $s$-group sparse if there exists $M \leq p$ and a partition $G_1,\cdots,G_M$  of $[p]$ into $M$ disjoint groups such that for some index set $I \subset [M]$ with $|I|\leq s$ we have that $\sum_{i=1}^M \| h_{G_i} \|_2 = \sum_{i \in I} \| h_{G_i} \|_2$. 

Henceforth we assume that $h^*$ is $s$-group sparse. 
We assume that the $X_i$'s are independent, isotropic Gaussian vectors with covariance matrix $I_p$. 
 The group sparsity assumption suggests to use the group Lasso norm proposed in~\cite{yuan2006model}. 
 The group Lasso norm is defined as
\begin{equation} \label{def group Lasso norm}
    \| h \|_{\textnormal{GL}} = \sum_{i=1}^M \| h_{G_i} \|_2.
\end{equation}
Theoretical results for the group Lasso are established in~\cite{lounici2011oracle} and show that $\|\cdot\|_{\textnormal{GL}}$ is indeed the right penalty to promote group sparsity. Here, we consider the interpolating solution with minimal group Lasso norm defined as
\begin{equation} \label{def_groupLasso}
    \hat h \in \argmin_{h \in \mathbb R^p} \|h\|_{\textnormal{GL}}\quad \textnormal{subject to} \quad \langle X_i, h\rangle = Y_i, \quad i=1,\cdots, n . 
\end{equation}
Provided that the dimension $p$ is large enough and that the groups have similar and large enough sizes, we obtain again a prediction error bound that only depends on the average size of the errors.
\begin{theorem} \label{thm groupLasso} 
Define
$$
W = \frac{\max_{i \in [M]} | G_i |}{ \min_{i \in [M]} |G_i|}. 
$$
There exist constants $c_1,c_2,c_3>0$ such that if $p \geq c_1 n \log(n + W)$ and $\log(M) \leq  c_2 \min_{i \in [M]} | G_i| $, then with probability at least $1-Me^{-c_3n}$
\begin{align} \label{Theorem group Lasso no sparsity}
    \|\hat h -h^* \|_2^2 \lesssim \bigg(      W \frac{\|\xi\|_2^2}{n}+\|h^* \|_{\textnormal{GL}}^2 \frac{\max_{i \in [M]} |G_i| }{n} \bigg).
\end{align}
Moreover, if for some small enough constant $c_4 >0$ we have $s \max_{i \in [M]} |G_i| \leq c_4 n$, then on the same event
\begin{align} \label{Theorem group Lasso}
    \|\hat h -h^* \|_2^2 \lesssim  W \frac{\|\xi\|_2^2}{n}.
\end{align}
\end{theorem}
As for the $\ell_1$-norm, we observe a phase transition. When $s \max_{i} |G_i| \lesssim n$ and the group sizes are comparable, the Euclidean estimation error is bounded by the average noise level. By constrast, when $n \lesssim s \max_{i} |G_i|  $ the subdifferential condition (SD) is not fulfilled and an extra term depending on $\| h^*\|_{\textnormal{GL}}$ appears.

By applying again  Proposition \ref{Prop lower bd gen} with  $L=\{h \in \mathbb{R}^p: ~h~ \textnormal{is $s$-group sparse} \}$, we obtain that the error bound in \eqref{Theorem group Lasso} is minimax optimal for adversarial errors when $W$ is constant.

Moreover, applying Theorem \ref{thm RERM}, yields similar results for the group Lasso $\hat h_\lambda$ \cite{yuan2006model} with an additional error term depending on the regularization parameter $\lambda$. 

\begin{theorem} \label{Erm groupLasso} 
Under the same conditions as in Theorem~\ref{thm groupLasso} with probability at least $1-Me^{-c_1n}$ we have for $\hat h_\lambda$, $\lambda >0$, that
\begin{align*} 
    \|\hat h_{\lambda} -h^* \|_2^2 \lesssim  \left(      W \frac{\|\xi\|_2^2}{n} +\|h^* \|_{\textnormal{GL}}^2  \frac{\max_{i \in [M]} |G_i| }{n}  +  \lambda \| h^* \|_{\textnormal{GL}}\right) .
\end{align*}
Moreover, if for some small enough constant $c_2 >0$ we have $s \max_{i \in [M]} |G_i| \leq c_2 n$, then on the same event
\begin{align*}
    \|\hat h_{\lambda} - h^* \|_2^2 \lesssim \left( W \frac{\|\xi\|_2^2}{n} + s\lambda^2  \right).
\end{align*}
\end{theorem}
It is possible to extend Theorem \ref{thm groupLasso} and Theorem \ref{Erm groupLasso} to features that consist of independent zero mean sub-Gaussian random variables by substituting the lower tail bounds for $\chi^2$-random variables in the proof of Lemma \ref{lemma noise interpol group Lasso} by an application Bernstein's inequality and by using Corollary 2.6. in \cite{BoucheronLugosiMassart13} to bound the local Rademacher complexity. 

\subsection{Low rank trace regression} \label{subsec trace}
We now consider a linear system of matrix equations, which is called trace regression. Hence, $\mathbb{H}$ is a matrix valued space, $\mathbb{H}=\mathbb{R}^{p_1 \times p_2}$. 
Instead of sparsity, we assume that $h^*$ has low rank, $\text{rank}(h^*)\leq s$, and, similar to before, we assume that the $X_i$'s are isotropic Gaussian matrices, $X_{ikl} \overset{i.i.d.}{\thicksim} \mathcal{N}(0,1)$.  The analogue to basis pursuit is nuclear (Schatten-1) norm minimization defined as
\begin{align} \label{def min nu-cu-ler est}
   \hat h \in \argmin_{h \in \mathbb{R}^{p_1 \times p_2}} \|h\|_{S_1} ~~~\text{subject to}~~~ \langle X_i, h \rangle =Y_i, ~i=1, \dots, n,
\end{align}
where we recall $\| h \|_{S_1} := \sum_{i=1}^{\text{rank}(h)} \sigma_i(h)$ and where $(\sigma_i(h))_i$ denote the singular values of $h$. \\
This program was proposed by \cite{FazelHindiBoyd01} and shown to recover $h^*$ in the noiseless case as long as $s \max (p_1,p_2) \lesssim n$  \cite{RechtFazelParrilo10}. Moreover, the output of deep linear neural networks trained with gradient descent on trace regression data converges to \eqref{def min nu-cu-ler est} when trained long enough  \cite{AroraCohenHuLuo19}. 

Similar to the $\ell_1$-norm promoting sparsity, the nuclear norm promotes a low-rank and the subdifferential condition (SD) is fulfilled when $s \max(p_1,p_2) \lesssim n$. 
Hence, provided that the overparameterization is large enough compared to the number of samples and that the rank is small enough, we obtain again a quantitative error bound that only depends on the average size of the errors $\xi$.  
\begin{theorem} \label{corollary matrix}
There exist positive constants $c_1,c_2$ such that if $n \log \big(n(p_2+p_1) \big)  \leq c_1p_1p_2$, then, with probability at least $1- \exp(-c_2n)$
$$
\|  \hat h - h^* \|_2^2 \lesssim  \bigg(  \frac{\|\xi\|_2^2}{ n} + \| h^*\|_{S_1}^2  \frac{p_1 + p_2}{n} \bigg) . 
$$
Moreover, if for some small enough constant $c_3>0$ we have $s \max(p_1,p_2) \leq c_3 n $, then, on the same event
\begin{equation} \label{bound error matrix 2}
\|  \hat h - h^* \|_2^2 \lesssim  \frac{\|\xi\|_2^2}{ n}. 
\end{equation}
\end{theorem} \label{thm low rank upper}
We observe again a phase transition. When $ s \max(p_1, p_2) \lesssim  n$, the subdifferential condition (SD) is fulfilled and the Euclidean estimation error is bounded by the average squared noise in the overparameterized regime. By contrast, when $n \lesssim s \max(p_1, p_2)$  an extra term depending on $\| h^*\|_{S_1}$ appears. 

Applying Proposition \ref{Prop lower bd gen} with $L=\{h \in \mathbb{R}^{p_1 \times p_2}: ~\text{rank}(h) \leq s\}$, we obtain that the bound in \eqref{bound error matrix 2} is minimax optimal against adversarial errors for the class of rank $s$ matrices, $s \geq 1$. 

Moreover, applying Theorem \ref{thm RERM}, we obtain similar results for the RERM, the matrix Lasso, with an additional error term depending on $\lambda$.

\begin{theorem} \label{erm trace regression}
Under the same conditions as in Theorem~\ref{corollary matrix}, we have for $\hat h_\lambda$, $\lambda >0$, that with probability at least $1- \exp(-c_1n)$
$$
\|  \hat h_{\lambda} - h^* \|_2^2 \lesssim \bigg(  \frac{\|\xi\|_2^2}{ n} +\| h^*\|_{S_1}^2 \frac{p_1 + p_2}{n} + \lambda \| h^* \|_{S_1} \bigg) . 
$$
Moreover, if for some small enough constant $c_2>0$ we have $s \max(p_1,p_2) \leq c_2 n $, then, on the same event
\begin{equation*} 
\|  \hat h_\lambda - h^* \|_2^2 \lesssim \left( \frac{\|\xi\|_2^2}{ n} +  s \lambda^2  \right) . 
\end{equation*}
\end{theorem}
It is again possible to extend Theorem \ref{corollary matrix} and Theorem \ref{erm trace regression} to features that consist of independent zero mean sub-Gaussian random variables by substituting the lower tail bounds for $\chi^2$-random variables in the proof of Lemma \ref{Lemma lower bound spec} by an application Bernstein's inequality and by using standard  bounds for the expected spectral norm of sub-Gaussian random matrices to bound the local Rademacher complexity \cite{Vershynin12}.

\subsection{Linear model with Euclidean norm} \label{subsec regression ell2}

We now consider linear regression with Euclidean norm penalty. Here  $\mathbb{H}=\mathbb{R}^p$, but we remark that the results can be  extended to an infinite dimensional Hilbert space $\mathbb H$. Contrary to before, we do not assume low-dimensional intrinsic structure of $h^*$, but instead that the covariance matrix $\Sigma$ of the $X_i$'s has a decay structure as in ~\cite{BartlettLongLugosiTsigler20,ChinotLerasle20}. We assume that the $X_i$ are independent Gaussian random variables with covariance matrix $\Sigma$. The minimum Euclidean norm interpolating estimator is defined as
\begin{equation} \label{def mini euclidean}
    \hat h = \argmin_{ h \in \mathbb{H}} \|h \|_2 ~~~~~~\textnormal{subject to}~~~\langle X_i, h \rangle = Y_i, ~~i=1, \dots, n.
\end{equation}
For $i \in [p]$, we denote $\lambda_i(\Sigma)$ the $i$-th eigenvalue of $\Sigma$,  ordered in decreasing order  $\lambda_1(\Sigma) \geq \lambda_2(\Sigma) \geq \cdots \geq  \lambda_p(\Sigma)$, and for any $k$ in $[p]$ we define $r_k(\Sigma) = \sum_{i=k}^p \lambda_i(\Sigma)$. When applying Theorem \ref{Main general theorem} to this setting, we exactly recover the result of Theorem 1 in~\cite{ChinotLerasle20} showing the general applicability of Theorem~\ref{Main general theorem}.
\begin{theorem}[Corollary 1~\cite{ChinotLerasle20}] \label{corollary euclidean}
With the convention $\inf \emptyset = + \infty $, we define
$$
k^* = \inf \bigg\{k \in [p]: \frac{r_k(\Sigma)}{\lambda_k(\Sigma)} >c_1n \bigg\},
$$ 
for some constant $c_1>0$. Assume that for some constant $c_2 >0$ it holds that $k^* \leq c_2 n$. Then there exist positive constants $c_3,c_4$ such that with probability at least $1- \exp(-c_3n)$ the estimator $\hat h$ defined in \eqref{def mini euclidean} satisfies
\begin{equation} \label{bound lin reg l2}
    \| \Sigma^{1/2} (\hat h - h^*) \|_2^2 \lesssim  \frac{\| \xi \|_2^2}{n} +  \| h^* \|_2^2 \frac{ r_{c_4n}(\Sigma)}{n}   .
\end{equation}
\end{theorem}
Contrary to Theorems~\ref{thm cs}, \ref{thm groupLasso} and~\ref{corollary matrix}, there is no phase transition and the upper bound in Theorem~\ref{corollary euclidean} always has an extra term depending on the Euclidean norm of $h^*$. This extra term comes from the fact that the subdifferential condition (SD) is not satisfied and only the bound \eqref{thm main bound 1} of Theorem~\ref{Main general theorem} can be applied. This extra term is likely to be unavoidable since in the noise-free setting exact recovery is impossible without further intrinsic low-dimensional structure of $h^*$.

When $\|h^*\|_2^2 r_{c_3 n}(\Sigma) \lesssim \|\xi\|_2^2$, the first term in \eqref{bound lin reg l2} dominates. Assuming additionally that $r_{c_3 n}(\Sigma) \lesssim n$, we can apply Proposition \ref{Prop lower bd gen} with $L=\{h \in \mathbb{R}^p: ~\|h\|_2^2 r_{c_3 n}(\Sigma) \leq  c \|\xi\|_2^2\}$ for some constant $c>0$ and obtain that the bound \eqref{bound lin reg l2} is for this class minimax optimal against adversarial errors. 
A stronger, pointwise in $h^*$ and $\xi$ lower bound  holding with high probability is established in~\cite{ChinotLerasle20}, showing also that the term $\| \xi \|_2 / \sqrt n$ cannot be improved. 

Applying Theorem \ref{thm RERM} to the ridge estimator $\hat h_\lambda$, $\lambda >0$, it is also possible to obtain a similar bound as in \eqref{bound lin reg l2} with an additional additive term of order ${\lambda \|h^*\|_2}$ in the upper bound.

\section{Discussion and concluding remarks} \label{subsec concl}
We have given a unified perspective on the performance of minimum norm interpolating algorithms and regularized empirical risk minimizers (RERM) in regression type problems with adversarial errors. When the norm induces low-dimensional intrinsic structure, minimum norm interpolators, despite interpolating the data, were shown to achieve the same convergence rates as state of the art regularized  algorithms \cite{CandesRombergTao06} for worst case adversarial errors. In the basis pursuit case this gives an alternative, more general and  constructive proof of the stability result of \cite{Wojtaszczyk10}. Similarly, RERM achieves optimal convergence rates for worst case adversarial errors as long as the regularization parameter is picked small enough.

Compared to RERM and other regularized algorithms that are robust to adversarial errors such as regularized basis pursuit \cite{BrugiapagliaAdcock18}, we conclude that interpolating algorithms have the  benefit that they do not require  careful choice of tuning parameters. Instead, their performance automatically adapts to the size of the errors, whereas RERM and regularized basis pursuit require that tuning parameters are chosen small enough in order to achieve optimal performance.  

However, in the presence of low-dimensional intrinsic structure and well-behaved noise that is symmetric, sub-Gaussian and independent of the $X_i$ and $h^*$, RERM 
with optimal, large enough, tuning parameter achieves a much smaller prediction error than minimum norm interpolated estimators \cite{lecue2018regularization}. For example, for sparse linear regression with i.i.d. standard Gaussian design and Gaussian noise with variance $\sigma^2$, RERM achieves a prediction error  of order $\sigma^2 s\log(p/s)/n$ versus $\sigma^2$ for the minimum $\ell_1$-norm interpolator. This gap is not due to our proof techniques but is inherent to the nature of using an interpolating solution as shown in~\cite{MuthukumarVodrahalliSubramanianSahai20}.

Below we present some open problems and possible extensions of our theory. \\

\noindent \textbf{Bounds for $\|\hat \nu\|$ and $\inf_{v \in \mathcal{S}^{n-1}}\| \sum v_i X_i \|^*$.} It would be interesting to obtain general bounds for $\|\hat \nu\|$ and $\inf_{v \in \mathcal{S}^{n-1}}\| \sum v_i X_i \|^*$, even if one imposes a Gaussianity assumption on the $X_i$. In our examples we had to rely on a slightly different technique in each case to obtain a lower bound and it is not clear whether a general technique for arbitrary norms $\|\cdot\|^*$ is available. For the particular case of $\|\cdot\|=\|\cdot\|_1$ bounds for $\|\hat \nu\|$ and $\inf_{v \in \mathcal{S}^{n-1}}\| \sum v_i X_i \|^*$ are closely connected to the geometry of random polytopes and were studied under the name $\ell_1$-quotient property \cite{PanelaLitvakaPajorRudelsonJaegermann05,DevorePetrovaWojtaszczyk09,Wojtaszczyk10,KrahmerKuemmerleRauhut18,GuedonKrahmerKuemmerleMendelsonRauhut19}. It would be highly interesting to investigate which geometrical structures are the correct analagons for other norms and how to study their geometry. \\

\noindent \textbf{Matrix completion:} We expect that it should be possible to extend our results to noisy matrix completion. In particular,  it would be necessary to consider sampling without replacement as in \cite{CandesTao10}. This does not fit exactly in our framework, as we assume in Theorem \ref{Main general theorem} that the $X_i$ are independent. When sampling with replacement (e.g. \cite{Recht}), there is at least one entry that is sampled twice with high probability, leading to the non-existence of a solution of the minimal interpolating nuclear norm objective \eqref{def min nu-cu-ler est} in the presence of errors. \\

 \noindent \textbf{Lower bounds:} Proposition \ref{Prop lower bd gen} is a worst case result for worst case type, adversarial errors. It would be highly interesting to derive lower bounds for minimum norm interpolated estimators that hold in more optimistic scenarios, for instance when the $\xi_i$'s are independent Gaussians with mean zero. In~\cite{MuthukumarVodrahalliSubramanianSahai20}, the authors consider sparse linear regression with minimum $\ell_1$-norm interpolation, basis pursuit. They show that when $h^* =0$, $X_i \overset{i.i.d.}{\thicksim} \mathcal{N}(0, I_p)$ and $\xi_i \overset{i.i.d.}{\thicksim} \mathcal{N}(0,\sigma^2)$, basis pursuit fulfills with high probability $$\frac{\|\xi\|_2^2}{ n\log(p/n)} \lesssim \| \hat h  \|_2^2  .$$ 

 Hence, in this case the bound \eqref{bound cs master} in Theorem \ref{thm cs} is optimal up to a logarithmic factor. We conjecture that this logarithmic gap does not exist and that  $\| \xi \|_2^2 /n$ is the correct lower bound
 and we leave this open as an important question for future research. 
\section{Proofs} \label{sec proofs} 

\subsection{Proof of Theorem~\ref{Main general theorem}}
\begin{proof}
For $\gamma >0$ we define
$$
\mathcal C_{r^*(\gamma)} = \{ h \in \mathbb H: \| \langle X, h \rangle \|_{L_2(\mu)} \geq r^*(\gamma) \| h \| \}.
$$
The proof follows from the following two propositions:

\begin{proposition} \label{thm_small_ball_gen}
Assume that the small-ball assumption holds with $\kappa,\delta >0$. Then with probability at least $1 -2\exp(-\delta^2 n /16)$, $\hat h$ satisfies
$$
\mathbf 1 \left (\hat h- h^* \in C_{r^*(\delta \kappa/32)} \right  ) \| \langle X, \hat h - h^* \rangle \|_{L_2(\mu)}^2 \leq \frac{8}{\delta \kappa^2 } \frac{\| \xi \|_2^2}{  n}   .
$$
\end{proposition}

\begin{proposition} \label{Sparsity theorem}

We have that 
$$
\mathbf 1 \left (\hat h- h^*  \notin C_{r^*(\delta \kappa/32)} \right ) \| \langle X, \hat h - h^* \rangle \|_{L_2(\mu)}  \leq r^*(\delta \kappa/32) \big( 2 \|h^* \| + \| \hat \nu \| \big). 
$$
Moreover, if there exists $\zeta > 0$ such that $\Delta(\delta \kappa/32, h^* ) \geq \zeta$, then 
$$
\mathbf 1 \left (\hat h- h^* \notin C_{r^*(\delta \kappa/32)}\right  ) \| \langle X, \hat h - h^* \rangle \|_{L_2(\mu)}  \leq r^*(\delta \kappa/32) \frac{\| \hat \nu \|}{\zeta}. 
$$
\end{proposition}
\end{proof}

\subsubsection{Proof of  Proposition~\ref{thm_small_ball_gen}}
\begin{proof}
The proof is based on the small-ball method  \cite{koltchinskii2015bounding,mendelson2014learning}.
We define $\gamma = \delta \kappa /32$.  Since $ \hat  h$ is an interpolator, we have that 
\begin{equation} \label{small_ball_start}
    \frac{\| \langle X, \hat h -h^* \rangle \|_{L_2(\mu)}^2}{r^*(\gamma)^2} \frac{1}{n}\sum_{i=1}^n \langle X_i,  \frac{r^*(\gamma)(\hat h - h^*)}{\| \langle X, \hat h -h^* \rangle \|_{L_2(\mu)}} \rangle^2 = \frac{1}{n}\sum_{i=1}^n \xi_i^2 .
\end{equation}
Let $\tilde h = r^*(\gamma) (\hat h - h^*) / \| \langle X, \hat h -h^* \rangle \|_{L_2(\mu)}$. When $\hat h - h^*\in C_{r^*(\gamma)}$, we have that $\| \langle X, \tilde h \rangle \|_{L_2(\mu)} = r^*(\gamma)$ and $\| \tilde h \| \leq 1$. It follows that
$$
\mathbf{1} \left ( \hat h - h^*\in C_{r^*(\gamma)} \right )\| \langle X, \hat h -h^* \rangle \|_{L_2(\mu)}^2 \leq \frac{\| \xi \|_2}{n} \frac{r^*(\gamma)^2}{\inf_{\substack{ h:\| \langle X, h \rangle \|_{L_2(\mu)} = r^*(\gamma) \\ \| h \| \leq 1} } \frac{1}{n}\sum_{i=1}^n \langle X_i, h \rangle^2 }.
$$
The small-ball method consists of showing that there exists $\nu >0$ such that
\begin{equation} \label{goal_small_ball}
     \inf_{\substack{ h:\| \langle X, h  \rangle \|_{L_2(\mu)} = r^*(\gamma) \\ \| h \| \leq 1} }\frac{1}{n}\sum_{i=1}^n \langle X_i, h \rangle^2 \geq \nu r^*(\gamma)^2
\end{equation}
holds with high probability. 
Equation \eqref{goal_small_ball} holds if for some $\eta > 0$ 
\begin{equation} \label{goal_small_ball_2}
    A_n := \inf_{\substack{ h:\| \langle X, h  \rangle \|_{L_2(\mu)} = r^*(\gamma) \\ \| h \| \leq 1} }\sum_{i=1}^n \mathbf 1 \big(  \big | \langle X_i , h \rangle \big| \geq \eta r^*(\gamma) \big ) \geq \frac{\nu n }{\eta^2} . 
\end{equation}
Hence, it suffices to show that \eqref{goal_small_ball_2} holds with high probability. 
Let $\phi: \mathbb R \mapsto \mathbb R$ be defined as
$$
\phi(t) = \left\{
    \begin{array}{ll}
       0 & \mbox{if } t \leq 1 \\
       t-1 & \mbox{if } t \in [1,2]\\
       1 & \mbox{if } t \geq 2
    \end{array}
\right.
$$
The function $\phi$ is $1$-Lipschitz and satisfies $\mathbf 1(t \geq 2) \leq \phi(t) \leq \mathbf 1(t \geq 1) $ for every $t \in \mathbb R$. For $h \in \mathbb H$ such that $\| \langle X,h \rangle \|_{L_2(\mu)} = r^*(\gamma)$ and $\| h \| \leq 1$, we obtain that 
\begin{align*}
&   \frac{1}{n} \sum_{i=1}^n \mathbf 1 \big( \big| \langle X_i, h \rangle \big| \geq \eta r^*(\gamma)\big )   \geq \frac{1}{n}\sum_{i=1}^n \phi \left ( \frac{\big| \langle X_i, h \rangle \big|}{\eta r^*(\gamma)} \right )\\
    \geq & \mathbb{E} \left [\phi \left ( \frac{\big| \langle X_1, h \rangle  \big|}{\eta r^*(\gamma)} \right ) \right ]- \left |  \frac{1}{n}\sum_{i=1}^n \phi \left ( \frac{\big| \langle X_i, h \rangle  \big|}{\eta r^*(\gamma)} \right )- \mathbb{E} \left [ \phi \left ( \frac{\big| \langle X_i, h \rangle  \big|}{\eta r^*(\gamma)} \right ) \right ] \right | \\
    \geq & \mathbb P \big( \big| \langle X_1, h \rangle  \big| \geq 2 \eta r^*(\gamma)  \big)  
  - \left |  \frac{1}{n}\sum_{i=1}^n \phi \left ( \frac{\big| \langle X_i, h \rangle \big|}{\eta r^*(\gamma)} \right )- \mathbb{E} \left [ \phi \left ( \frac{\big| \langle X_i, h \rangle  \big|}{\eta r^*(\gamma)} \right ) \right ] \right | . 
\end{align*}

It follows that 
\begin{align*}
A_n & \geq \inf_{\substack{h: \| \langle X, h \rangle\|_{L_2(\mu)}  = r^*(\gamma) \\ \| h \| \leq 1}} \mathbb P \big( \big| \langle X_1, h \rangle \big| \geq 2 \eta \| \langle h, X_1 \rangle \|_{L_2(\mu)} \big) \\
& - \sup_{\substack{ h:\| \langle X, h \rangle\|_{L_2(\mu)}  = r^*(\gamma) \\ \| h \| \leq 1}}  \frac{1}{n} \bigg| \sum_{i=1}^n \phi \bigg( \frac{|\langle X_i,h \rangle|}{\eta r^*(\gamma) }\bigg) - \mathbb E \bigg[ \phi \bigg( \frac{|\langle X_i,h \rangle|}{\eta r^*(\gamma)}\bigg) \bigg]  \bigg| . 
\end{align*}
Taking $2\eta = \kappa$, we obtain by applying the small-ball assumption that 
$$
\inf_{\substack{ h:\| \langle X, h \rangle\|_{L_2(\mu)}  = r^*(\gamma) \\ \| h \| \leq 1}} \mathbb P \big( \big| \langle X_1, h \rangle \big| \geq 2 \eta \| \langle h, X_1 \rangle \|_{L_2(\mu)} \big) \geq \delta. 
$$
Moreover, since $|\phi(t)| \leq 1$ we obtain by applying the bounded differences inequality, Theorem 3.3.14 in \cite{GineNickl16}, that with probability at least $1-t$
\begin{align*}
    & \sup_{\substack{ h:\| \langle X, h \rangle\|_{L_2(\mu)}  = r^*(\gamma) \\ \| h \| \leq 1}}  \frac{1}{n} \bigg| \sum_{i=1}^n \phi \bigg( \frac{|\langle X_i,h \rangle|}{\eta r^*(\gamma) }\bigg) - \mathbb E \bigg[ \phi \bigg( \frac{|\langle X_i,h \rangle|}{\eta r^*(\gamma)}\bigg) \bigg]  \bigg|  \\
    & \leq \mathbb E \sup_{\substack{h: \| \langle X, h \rangle\|_{L_2(\mu)}  = r^*(\gamma) \\ \| h \| \leq 1}}  \frac{1}{n} \bigg| \sum_{i=1}^n \phi \bigg( \frac{|\langle X_i,h \rangle|}{\eta r^*(\gamma) }\bigg) - \mathbb E \bigg[ \phi \bigg( \frac{|\langle X_i,h \rangle|}{\eta r^*(\gamma)}\bigg) \bigg]  \bigg|  +  \sqrt \frac{\log(2/t)}{n} \\
    & \leq \frac{4}{\eta r^*(\gamma)} \mathbb E \sup_{\substack{ h:\| \langle X, h \rangle\|_{L_2(\mu)}  = r^*(\gamma) \\ \| h \| \leq 1}} \frac{1}{n} \sum_{i=1}^n \varepsilon_i \langle X_i ,h \rangle + \sqrt \frac{\log(2/t)}{n},
\end{align*}
where $(\varepsilon_i)_{i=1}^n$ are i.i.d. Rademacher random variables that are independent from $(X_i)_{i=1}^n$. For the second inequality we used first a symmetrization inequality and then a contraction inequality,  Theorem 3.1.21 and Theorem 3.2.1. in \cite{GineNickl16}, respectively. 
Using the definition of $r^*(\gamma)$, we have  with probability at least $1-t$ that 
$$
A_n \geq \delta - \frac{4 \gamma}{\eta} - \sqrt{\frac{\log(2/t)}{n}}.
$$
Thus, since $2 \eta = \kappa$, choosing $\nu = \delta \kappa^2/8$, $\gamma = \kappa \delta /32$ and $t = 2\exp(-\delta^2 n /16)$  we obtain that 
$$
A_n \geq \frac{\nu}{\eta^2} . 
$$
Hence, overall we obtain that with probability at least $1- 2\exp(-\delta^2 n /16)$
$$
 \mathbf{1} \left (\hat h - h^*\in C_{r^*(\gamma)} \right )\| \langle X , \hat h - h^* \rangle \|_{L_2(\mu)}^2 \leq \frac{8}{\delta \kappa^2 } \frac{\| \xi \|_2^2}{ n} ,
$$
concluding the proof. 
\end{proof}
\subsubsection{Proof of Proposition~\ref{Sparsity theorem}}
\begin{proof}
We set $\gamma = \delta \kappa /32$ and throughout assume that $\hat h-h \notin \mathcal{C}_{r^*(\gamma)}$ because otherwise the claimed bound is trivial. By definition of $\hat \nu$, we have that $\langle X_i , h^* + \hat \nu \rangle = Y_i$ and it follows that 
$$
0 \geq \|\hat h \| - \| h^* + \hat \nu \| \geq \| \hat h - h^* \| - 2 \| h^* \| - \| \hat \nu \| . 
$$
Since $\hat h - h^* \notin \mathcal{C}_{r^*(\gamma)}$, this shows the first statement of Proposition~\ref{Sparsity theorem}. \\
For the second claim, consider $g \in \partial (\| \cdot \|)_{h^*}$. By definition of the subdifferential, we obtain that 
\begin{equation} \label{eq1 proof se}
    \| \hat h \| - \| h^* \| \geq \langle g , \hat h -h^* \rangle  = \| \hat h - h^* \| \langle g, \frac{\hat h -h^*}{\| \hat h - h^* \|} \rangle . 
\end{equation}
Let $\tilde h = \hat h -h^* / \| \hat h - h^* \|$. We have $\| \tilde h \| = 1$ and, since $(\hat h - h^*) \notin \mathcal{C}_{r^*(\gamma)}$,  $\| \langle X, \tilde h \rangle \|_{L_2(\mu)} \leq r^*(\gamma)$. Hence, and since~\eqref{eq1 proof se} holds for any $g \in \partial (\| \cdot \|)_{h^*}$, we obtain by using the subdifferential condition that  
$$
0 \geq \| \hat h \| - \| h^* \| - \| \hat \nu \| \geq  \| \hat h - h^* \| \Delta(\gamma, h^*) - \| \hat \nu \| \geq \zeta \| \hat h - h^* \| - \| \hat \nu \| .
$$
Recalling that $\hat h - h^* \notin \mathcal{C}_{r^*(\gamma)}$ concludes the proof.
\end{proof}

\subsection{Proof of Lemma~\ref{duality formulation}}\label{sec_duality}
\begin{proof}
We use Lagrangian duality to control $\| \hat \nu \|$, where we recall that $\hat \nu$ is defined as the solution of 
\begin{equation} \label{primal_problem}
    \inf_{h \in \mathbb H} \| h \| \quad \textnormal{subject to} \quad  \langle h,X_i \rangle = \xi_i, \quad i=1,\cdots,n.
\end{equation}
For $v\in \mathbb R^n $, we  define the Lagrangian $\mathcal L: \mathbb H \times \mathbb R^n \mapsto \mathbb R$ as
$$
\mathcal L(h,v) = \| h \| + \sum_{i=1}^n v_i \big( \langle h,X_i \rangle - \xi_i   \big) .
$$
The dual problem of \eqref{primal_problem} is defined as
\begin{equation} \label{dual_problem}
    \sup_{v \in \mathbb R^n} \inf_{h \in \mathbb H} \mathcal L(h,v).
\end{equation}
We have that 
$$
\inf_{h \in \mathbb H} \mathcal L(h,v) = \inf_{h \in \mathbb H} \big \{  \| h \| + \sum_{i=1}^n v_i \big( \langle h,X_i \rangle - \xi_i   \big) \big \} = -v^T \xi - \sup_{h \in \mathbb H} \big \{ -   \langle h, \sum_{i=1}^n v_i X_i \rangle  - \| h\|   \big\} ,
$$
where $\xi = (\xi_1,\cdots \xi_n)$. Recall that for any function $f: \mathbb H \mapsto \mathbb R$, the conjugate $f^*$ is defined as
\begin{equation} \label{def_conjugate}
    f^*(y) = \sup_{x \in \mathbb H} \big \{ \langle x,y \rangle - f(x)  \big \} .
\end{equation}
In particular (see \cite{boyd2004convex}, Example 3.26), when $f(h) = \|h\|$, for $\| \cdot \|$ a norm acting on $\mathbb H$, we have that 
\begin{equation} \label{conjugate_norm}
f^*(y) = \left\{
    \begin{array}{ll}
        0 & \mbox{if } y \in B^*  \\
        \infty & \mbox{otherwise} ,
    \end{array}
\right.
\end{equation}
where $B^*$ is the unit ball with respect to the dual norm of $\| \cdot \|$. 
From \eqref{def_conjugate} and~\eqref{conjugate_norm}, the dual problem \eqref{dual_problem} can be rewritten as
$$
\sup_{v \in \mathbb R^n} - v^T \xi \quad \textnormal{subject to} \quad \sum_{i=1}^n v_i X_i \in B^* .
$$
Since the $X_i$ are linearly independent and $\dim(\mathbb{H})\geq n$ by assumption , the Moore-Penrose inverse exists, and hence there exists  $h$ in $\mathbb H$ such that $\langle X_i,h \rangle = \xi_i$ for $i=1,\cdots,n$. Hence, Slater's condition (e.g. p. 226 in\cite{boyd2004convex}) holds and consequently there is no duality gap. It follows that 
\begin{align}
\| \hat \nu \| = \sup_{v \in \mathbb R^n} v^T \xi \quad \textnormal{subject to} \quad \|\sum_{i=1}^n v_i X_i \|^* \leq 1  . \label{dual problem proof}
\end{align} 
Applying the Cauchy-Schwarz inequality concludes the proof of the upper bound in \eqref{bound lemma dual}. The lower bound follows by choosing  $v= \frac{\xi}{\| \sum_{i=1}^n \xi_i X_i \|^*}$ in \eqref{dual problem proof}. 
\end{proof}

\subsection{Proof of Theorem~\ref{thm RERM}}
\begin{proof}
The proof of Theorem~\ref{thm RERM} is similar to the one of Theorem~\ref{thm cs}. We first present the following two lemmas, which will be used in the proof. 
\begin{lemma} \label{lemma control norm erm}
For $\lambda>0$ the estimator $\hat h_{\lambda}$ satisfies
$$
\| \hat h_{\lambda}\| \leq \| h^* \| +\|\hat \nu\|.
$$
\end{lemma}
\begin{lemma} \label{lemma sum square}
We have that 
$$
\sum_{i=1}^n \langle X_i, \hat h_{\lambda} - h^* \rangle^2 \leq 4 \| \xi \|_2^2 + 4 \lambda n  \left( \| h^* \| - \| \hat h_{\lambda}\| \right)
$$
\end{lemma}
The proof of Lemmas \ref{lemma control norm erm} and \ref{lemma sum square} will be given below the proof of Theorem \ref{thm RERM} based on these two lemmas. For the latter we argue as follows. \\

For $\gamma = \delta \kappa/32$ we define
$$
\mathcal C_{r^*(\gamma)} = \{ h \in \mathbb H: \| \langle X, h \rangle \|_{L_2(\mu)} \geq r^*(\gamma) \| h \| \}.
$$
Arguing exactly as in the proof of Proposition~\ref{thm_small_ball_gen}, we have with probability at least $1 -2\exp(-\delta^2 n /16)$
\begin{align} \label{eq small ball ineq rerm}
\mathbf 1 \left ( \hat h_{\lambda}- h^* \in C_{r^*(\gamma)} \right  ) \| \langle X, \hat h_{\lambda} - h^* \rangle \|_{L_2(\mu)}^2 \leq \frac{8}{\delta \kappa^2 } \frac{1 }{  n} \sum_{i=1}^n \langle X_i, \hat h_{\lambda} - h^* \rangle^2.
\end{align}
Applying Lemma~\ref{lemma sum square} we obtain
$$
\mathbf 1 \left (\hat h_{\lambda}- h^*\in C_{r^*(\gamma)} \right  ) \| \langle X, \hat h_{\lambda} - h^* \rangle \|_{L_2(\mu)} \leq  \frac{2\sqrt 8}{\sqrt \delta \kappa} \left( \frac{\| \xi \|_2}{\sqrt n} + \sqrt{\lambda \| h^* \|}  \right)
$$
We now consider the case $\hat h_\lambda-h^*\notin  C_{r^*(\gamma)}$. By the triangle inequality and afterwards an application of Lemma \ref{lemma control norm erm}, we obtain that 
$$
\| \hat h_{\lambda} - h^* \| \leq \|h^*\|+\|\hat h_{\lambda}\| \leq  2  \|h^* \| + \|\hat \nu\|  ,
$$
and thus 
$$
\mathbf 1 \left (\hat h_{\lambda}- h^* \notin C_{r^*(\gamma)} \right  ) \| \langle X, \hat h_{\lambda} - h^* \rangle \|_{L_2(\mu)} \leq  r^*(\gamma) \left( 2 \|h^* \| + \|\hat \nu\|\right) .
$$

We now turn to the second part of Theorem~\ref{thm RERM}, where the two subdifferential conditions are satisfied and there exist $\zeta,\bar \zeta>0$ such that $\Delta(\gamma,h^*) \geq \zeta$ and $\bar \Delta(\gamma,h^*) \leq \bar \zeta$. \\
When $\hat h_{\lambda} - h^* \notin C_{r^*(\gamma)} $ we use the lower bound on $\Delta(\gamma,h^*)$. Indeed, when $\hat h_{\lambda} - h^* \notin C_{r^*(\gamma)} $ there exists some $g \in \partial \left( \| \cdot \| \right)_{h^*}$ such that we have
\begin{align} \label{eq lower bound subd}
\| \hat h_{\lambda} \| - \|h^* \| \geq \left \langle g , \frac{\hat h_{\lambda} - h^*}{\|\hat h_{\lambda} - h^* \|}  \right \rangle \|\hat h_{\lambda} - h^* \| \geq \zeta \|\hat h_{\lambda} - h^* \| . 
\end{align}
Applying Lemma~\ref{lemma control norm erm} we then obtain
\begin{align*}
\|\hat h_{\lambda} - h^* \| \leq \frac{\| \hat h_{\lambda} \| - \|h^*\|}{\zeta} \leq \frac{\|\hat \nu\|}{\zeta}  .
\end{align*}
Hence,  using the above when  $\hat h_{\lambda} - h^* \notin C_{r^*(\gamma)}$,  we obtain 
\begin{align*}
\mathbf 1 \left ( \hat h_{\lambda}- h^* \notin C_{r^*(\gamma)} \right  ) \| \langle X, \hat h_{\lambda} - h^* \rangle \|_{L_2(\mu)} & \leq \mathbf 1 \left ( \hat h_{\lambda}- h^* \notin C_{r^*(\gamma)} \right  ) r^*(\gamma)  \|\hat h_\lambda - h^*\| \\ & \leq  \frac{r^*(\gamma) \|\hat \nu\|} {\zeta} .
\end{align*}
When $\hat h_{\lambda} - h^* \in C_{r^*(\gamma)} $ we use the upper bound on $\bar \Delta(\gamma,h^*)$. Indeed, in this case there exists some $g \in \partial \left( \| \cdot \| \right)_{h^*}$ such that
$$
 \|h^*\|- \| \hat h_{\lambda} \| \leq \left \langle g , \frac{r^*(\gamma)(h^*-  \hat h_{\lambda})}{\| \langle \hat h_{\lambda} - h^*, X \rangle \|_{L_2(\mu)}} \right \rangle \frac{\| \langle \hat h_{\lambda} - h^* , X \rangle \|_{L_2(\mu)}}{r^*(\gamma)} \leq \bar \zeta \frac{\| \langle \hat h_{\lambda} - h^*, X \rangle \|_{L_2(\mu)}}{r^*(\gamma)}.
$$
Finally, using \eqref{eq small ball ineq rerm}, applying Lemma~\ref{lemma sum square} and using the above, it follows that with probability at least $1-2\exp(-\delta^2n/16)$ 
\begin{align*}
& \mathbf 1 \left (\hat h_{\lambda}- h^* \in C_{r^*(\gamma)} \right  )\| \langle  \hat h_{\lambda} - h^*, X \rangle \|_{L_2(\mu)} \\ \leq & \frac{2\sqrt 8}{\sqrt \delta \kappa } \frac{\| \xi \|_2}{\sqrt n} + \mathbf 1 \left (\hat h_{\lambda}- h^* \in C_{r^*(\gamma)} \right  ) \frac{2\sqrt 8}{\sqrt \delta \kappa} \sqrt{ \frac{\lambda \bar \zeta}{r^*(\gamma)}} \| \langle  \hat h_{\lambda} - h^* ,X\rangle \|_{L_2(\mu)}^{1/2} .
\end{align*}
Using the inequality $ab \leq a^2/2 + b^2/2$ with $a = \left((32\lambda \bar \zeta)/(r^*(\gamma)\delta \kappa^2)\right)^{1/2}$ and \\$b = \mathbf 1 \left ( \hat h_{\lambda}- h^* \in C_{r^*(\gamma)} \right  )\| \langle  \hat h_{\lambda} - h^*,X \rangle \|_{L_2(\mu)}^{1/2}$ we obtain the final result. 
\end{proof}

\subsubsection{Proof of Lemma~\ref{lemma control norm erm}}
\begin{proof}
Since  the empirical risk is always greater or equal than zero, $\hat h_\lambda$ minimizes the RERM objective \eqref{def RERM} and $h^* + \hat{\nu}$ interpolates the data, we have that
\begin{align*}
   \lambda \|\hat h_\lambda\| & \leq  \frac{1}{2n} \sum_{i=1}^n \left ( Y_i-\langle X_i, \hat h_\lambda \rangle\right )^2 + \lambda \|\hat h_\lambda\| \leq \frac{1}{2n}\sum_{i=1}^n \bigg ( Y_i-\langle X_i, h^* + \hat \nu \rangle\bigg )^2 + \lambda \|h^* + \hat \nu\| \\ & = \lambda \|h^* + \hat \nu\| \leq  \lambda\|h^*\|+\lambda\|\hat \nu\|.
\end{align*}
Dividing by $\lambda, \lambda>0$, yields the result.
\end{proof}

\subsubsection{Proof of Lemma~\ref{lemma sum square}}
\begin{proof}
By definition of $\hat h_{\lambda}$ we have 
\begin{equation} \label{eq def minimum}
    \frac{1}{n} \sum_{i=1}^n \left( Y_i - \langle \hat h_{\lambda} ,X_i \rangle \right)^2 + 2\lambda \| \hat h_{\lambda} \| \leq     \frac{1}{n} \sum_{i=1}^n \bigg( Y_i - \langle  h^* ,X_i \rangle \bigg)^2 + 2\lambda \|  h^* \| .
\end{equation}
It follows from Equation~\eqref{eq def minimum}, that 
\begin{align*}
     \sum_{i=1}^n \langle \hat h_{\lambda} -h^* ,X_i \rangle^2  &\leq 2 \sum_{i=1}^n \xi_i \langle \hat h_{\lambda} - h^* , X_i \rangle + 2\lambda n \left( \| h^* \| - \| \hat h_{\lambda}\| \right)  \\
     & \leq 2 \| \xi \|_2 \left( \sum_{i=1}^n \langle \hat h_{\lambda} - h^* , X_i \rangle^2 \right)^{1/2} +  2\lambda n \left( \| h^* \| - \| \hat h_{\lambda}\| \right) .
\end{align*}
Using the inequality $ab \leq a^2/2 + b^2/2$  for $ a = 2 \| \xi \|_2$ and $b= \left( \sum_{i=1}^n \langle \hat h_{\lambda} - h^* , X_i \rangle^2 \right)^{1/2}$ concludes the proof. 

\end{proof}
\subsection{Proofs of Theorem~\ref{thm cs} and Theorem~\ref{erm lasso}}
\begin{proof}
The proof of Theorem~\ref{thm cs} is based on Theorem~\ref{Main general theorem} and Lemma~\ref{duality formulation}. We apply the general routine presented in Section~\ref{sec main}. 
\begin{itemize}
    \item[]\textbf{Step 1:} For $\mathbb H = \mathbb R^p$, the norm defined on $\mathbb H$ is the $\ell_1$-norm,  $\|h\|=\| h \|_1$. \\
    \item[] \textbf{Step 2:} Since $X_1 \sim \mathcal N(0,\Sigma)$, for any $h \in \mathbb{R}^p$, $\langle h , X_1 \rangle \sim \mathcal N(0,\|\Sigma^{1/2} h\|_2^2)$ and $\| \langle h , X_1 \rangle \|_{L_4(\mu)}^4 =  3 \|\Sigma^{1/2} h\|_2^4= 3 \|\langle h, X_1\rangle\|_{L^2(\mu)}^4$. Hence, by Lemma~\ref{paley argument}, the small-ball assumption is verified for any $\kappa$ in $[0,1)$ and $\delta = (1-\kappa^2)/3$. \\
    \item[] \textbf{Step 3:} By symmetry of $X_i$, the random variable $\varepsilon_i X_i$ is distributed as $X_i$ for any $i=1,\cdots,n$ and for any $h \in \mathbb R^p$, $n^{-1/2}\sum_{i=1}^n \varepsilon_i \langle X_i,h \rangle$ is distributed as $g^T \Sigma^{1/2}h$, where $g^T \sim \mathcal{N}(0,I_p)$. Let $r>0$. Observe that $\Sigma^{1/2} B_1$ is the convex hull of $2p$ points, namely $\{\pm \Sigma^{1/2} e_i\}_{i=1}^p$, and that $\max_i \|\Sigma^{1/2} e_i\|_2= \max_i \sqrt{ \Sigma_{ii}}$. 
    Hence, applying Proposition 1 in~\cite{bellec2019localized} (with $M=2p$, $T=\Sigma^{1/2} B_1 /\max_i \sqrt{\Sigma_{ii}}$ and $s=r/(\max_i \sqrt{\Sigma_{ii}}$) there), we obtain that 
    \begin{align*}
    \mathbb E \sup_{h \in B_1^p \cap r \Sigma^{-1/2}B_2^p} g^T \Sigma^{1/2} h 
    & \leq 4   \sqrt{\max_{i} \Sigma_{ii} \log_+ \big(8ep r^2 /(\max_i \Sigma_{ii}) \big)} \\
    \end{align*}
    where $\log_+(\cdot) = \max(\log(\cdot),1)$ and $B_2^p$ and $B_1^p$ denote the unit $\ell_2$ and $\ell_1$ ball in $\mathbb{R}^p$, respectively. It follows that for any $\gamma >0$
    $$
    r^*(\gamma) \leq \inf \{ r>0: 4 \sqrt{\max_{i} \Sigma_{ii}\log_+ \big(8ep r^2/(\max_{i} \Sigma_{ii}) \big)} \leq \gamma \sqrt n r \}.
    $$
    Using the inequality $\log(x) \leq x - 1$,  we obtain that when $p/n \geq (384e)/\gamma^2$
    $$
    r^*(\gamma) \leq \frac{\sqrt{48 \max_{i} \Sigma_{ii}}}{\gamma} \sqrt \frac{\log(p/n)}{n}.
    $$
 
    \item[]\textbf{Step 4:} Since the dual norm of the $\ell_1$-norm is the $\ell_{\infty}$-norm, Lemma~\ref{duality formulation} yields
    $$
    \| \hat \nu \|_1 \leq \frac{\| \xi \|_2}{ \inf_{v \in \mathcal S^{n-1}} \| \sum_{i=1}^n v_i X_i \|_{\infty}} .
    $$
    The following lemma establishes a lower bound on the infimum above. 
    \begin{lemma} \label{lemma_interpolation_noise}
Suppose that $\alpha < 1$ is a constant such that 
$$
p \geq  n \max \left [  \left( \log \left( \frac{72n}{\alpha} \right) \sqrt{2\pi} \right)^{1/(1-\alpha)} ,  e^{1/\alpha} \right]. 
$$
Then, with probability at least
$$
1-p\exp(- n/2) - \exp \left( -\frac{e^{1-\alpha}}{2 \sqrt{2\pi}} n \right)
$$ 
we have that 
\begin{align*}
    \inf_{v \in \mathcal S^{n-1}} \left  \| \sum_{i=1}^n v_i X_i  \right \|_{\infty} \geq   \frac{\sqrt{ (\alpha/2)\log( p/ n)}}{\sup_{\|b\|_1=1} \|\Sigma^{-1/2} b\|_1}.
\end{align*}
\end{lemma}
The proof of Lemma~\ref{lemma_interpolation_noise} is presented in Appendix~\ref{lemma supp bp}. \\ 
    \item[]\textbf{Step 5:} We finally turn to the subdifferential condition. For any $I \subset [p]$ and $h \in \mathbb R^p$, let $P_I h$ be defined as $(P_I h)_i = h_i \mathbf 1\{i \in I \}$. Now, let $h \in \mathbb R^p$ such that $\| h \|_1 = 1$ and $\| h \|_2 \leq r^*(\gamma)$ and $I = \textnormal{supp}(h^*) = \{ i \in [p]: h_i^* \neq 0 \}$. Let us define $g \in \mathbb R^p$ as
    $$
    g_i = \left\{
    \begin{array}{ll}
        \textnormal{sign}(h^*_i) & \mbox{if } i \in I \\
        \textnormal{sign}(h_i) & \mbox{if } i \in I^c. 
    \end{array}
    \right.
    $$
    Since $\| g \|_{\infty} =1$ and $\langle g,h^* \rangle = \| h^* \|_1$, we have $g \in \partial (\| \cdot \|_1)_{h^*}$ and it follows that
    \begin{align*}
        \langle h, g \rangle = \langle P_I h, g \rangle + \langle P_{I^c} h, g \rangle \geq \|P_{I^c} h \|_1 -  \|P_{I} h \|_1 \geq \| h \|_1 - 2 \|P_{I} h \|_1 .
    \end{align*}
    If $3\|P_Ih\|_1 \leq \|P_{I^c}h\|_1$, we immediately obtain 
    \begin{align*}
        \langle h, g \rangle \geq \| h \|_1 - 2 \|P_{I} h \|_1 \geq \|h\|_1 - \frac{\|P_I h\|_1 +\|P_{I^c}h\|_1}{2} \geq \frac{1}{2}.
    \end{align*}
    Similarly, if $3\|P_Ih\|_1 \geq  \|P_{I^c}h\|_1$ we obtain, since we assumed the restricted eigenvalue condition with parameter $\psi$, that
    \begin{align*}
        \langle h, g \rangle \geq \|h\|_1 - 2\|P_Ih\|_1 \geq 1 - 2\sqrt{s}\|P_I h\|_2 \geq 1- \frac{2\sqrt{s}}{\psi}\|h\|_{L^2(\mu)} \geq 1- \frac{2 \sqrt{s} r^*(\gamma)}{\psi} \geq \frac{1}{2},
    \end{align*} when $s\log(p/n)/n \leq \psi^2\gamma^2/768$.  
\end{itemize}
Taking $\kappa = 1/\sqrt 3$, we obtain that $\gamma = 1/(144\sqrt 3)$ and applying Theorem~\ref{Main general theorem} concludes the proof of Theorem~\ref{thm cs}.\\

The proof of Theorem~\ref{erm lasso} follows exactly the same steps. The only difference is the upper bound on $\bar \Delta(\gamma,h^*)$. Take
 $$
    g_i = \left\{
    \begin{array}{ll}
        \textnormal{sign}(h^*_i) & \mbox{if } i \in I \\
        - \textnormal{sign}(h_i)  & \mbox{if } i \in I^c. 
    \end{array}
    \right.
    $$
 Since $\| g \|_{\infty} =1$ and $\langle g,h^* \rangle = \| h^* \|_1$, we have $g \in \partial (\| \cdot \|_1)_{h^*}$. For $h \in \mathbb H$ such that $\| \Sigma^{1/2} h \|_2 = r^*(\gamma)$ and $\| h\| \leq 1$ we have 
    \begin{align*}
        \langle h, g \rangle = \langle P_I h, g \rangle  + \langle P_I^c h, g \rangle \leq  \|P_{I} h \|_1 - \| P_{I^c} h \|_1 .
    \end{align*}
If $3 \| P_I h \|_1 \leq \| P_{I^c}h \|_1$ we have $\langle h, g \rangle \leq 0$. If $3 \| P_I h \|_1 \geq \| P_{I^c}h \|_1$, the restricted eigenvalue condition yields
$$
 \langle h, g \rangle \leq \|P_I h\|_1 \leq \sqrt s \| P_I h \|_2 \leq \frac{\sqrt sr^*(\gamma)}{\psi}.
$$
\end{proof}
\subsection{Proofs of Theorem~\ref{thm groupLasso} and Theorem~\ref{Erm groupLasso}}
\begin{proof}
The proof of Theorem~\ref{thm groupLasso} is based on Theorem~\ref{Main general theorem} and Lemma~\ref{duality formulation}. We apply the general routine presented in Section~\ref{sec main}. 
\begin{itemize}
    \item[] \textbf{Step 1:} For $\mathbb H = \mathbb R^p$, the norm defined on $\mathbb H$ is the group Lasso norm defined in~\eqref{def group Lasso norm}. \\
    \item[] \textbf{Step 2:} Since $X_1 \sim \mathcal N(0,I_p)$, for any $h \in \mathbb{R}^p$, $\langle h , X_1 \rangle \sim \mathcal N(0,\|h\|_2^2)$ and hence $\| \langle h , X_1 \rangle \|_{L_4(\mu)}^4 =  3\|h\|_2^4=3 \|\langle h, X_1 \rangle_{L_2(\mu)}^4$. Hence, by Lemma~\ref{paley argument}, the small-ball assumption is verified for any $\kappa$ in $[0,1)$ and $\delta = (1-\kappa^2)/3$. \\
    \item[] \textbf{Step 3:} By symmetry of $X_i$, the random variable $\varepsilon_i X_i$ is distributed as $X_i$ for any $i=1,\cdots,n$ and for any $h \in \mathbb R^p$, $n^{-1/2}\sum_{i=1}^n \varepsilon_i \langle X_i,h \rangle$ is distributed as $g^T h$, where $g^T \sim \mathcal{N}(0,I_p)$. The dual norm of the group Lasso norm is given by 
    $$
    \| h \|_{\textnormal{GL}}^* = \max_{i=1,\cdots,M} \| h_{G_i} \|_2 , \quad h \in \mathbb R^p.
    $$
  Hence, for $B_{\text{GL}}^p$ and $B_2^p$ denoting the unit group Lasso norm and $\ell_2$ norm ball, respectively, we obtain by applying H\"older's inequality and Jensen's inequality that 
    $$
    \mathbb E \sup_{h \in B_{\textnormal{GL}}^p \cap r B_2^p} g^T h \leq \mathbb E \max_{i=1,\cdots ,M} \| g_{G_i}\|_2 \leq \left( \mathbb E \max_{i=1,\cdots ,M} \| g_{G_i}\|_2^2 \right)^{1/2} .
    $$
    The random variables $\| g_{G_i}\|_2^2$ are independent chi-square distributed random variables with $|G_i|$ degrees of freedom each. By applying Example 2.7 in~\cite{BoucheronLugosiMassart13}, we obtain that 
    \begin{align*}
        \mathbb E \max_{i=1,\cdots ,M} \| g_{G_i}\|_2^2 & \leq \max_{i=1,\cdots,M} | G_i| + 2 \sqrt{\max_{i=1,\cdots,M} | G_i| \log(M)} + 2 \log(M) \\
        & \leq 5 \max_{i=1,\cdots,M} | G_i|   ,
    \end{align*}
    when $\log(M) \leq \max_{i=1,\cdots,M} | G_i| $. It follows that 
    $$
    r^*(\gamma) \leq \frac{\sqrt{5}}{\gamma} \sqrt \frac{\max_{i=1,\cdots,M} | G_i |}{n} .
    $$
    \item[] \textbf{Step 4:} By definition of the dual norm of the group Lasso norm, Lemma~\ref{duality formulation} gives
    $$
    \| \hat \nu \|_1 \leq \frac{\| \xi \|_2}{ \inf_{v \in \mathcal S^{n-1}} \max_{i=1,\cdots,M} \left \| \left ( \sum_{j=1}^n v_j X_j \right )_{G_i} \right \|_2 } ,
    $$
  The following lemma establishes a lower bound on the infimum term. 
    \begin{lemma} \label{lemma noise interpol group Lasso}
    Let 
    $$
    W = \frac{\max_{i=1,\cdots, M } | G_i|}{\min_{i=1,\cdots,M} | G_i| }.
    $$
    Suppose that $p \geq 32n \log \big( 6\sqrt 2 (2\sqrt n + \sqrt W)\big)$. Then, with probability at least 
\begin{align*}
    1-e^{-p/32}-e^{-n/2 + \log(M)}
\end{align*}
we have that
\begin{align*}
   \inf_{v \in \mathcal S^{n-1}} \max_{i=1,\cdots,M} \left \| \left ( \sum_{i=1}^n v_i X_i \right )_{G_i} \right \|_2 \geq  \frac{\sqrt{\min_i |G_i|}}{2\sqrt 2}. 
\end{align*}
\end{lemma}
The proof of Lemma~\ref{lemma noise interpol group Lasso} is provided in Appendix~\ref{supp group Lasso}. \\
    \item[] \textbf{Step 5:} We finally turn to the subdifferential condition. Let $h \in \mathbb R^p$ such that $\| h \|_{\textnormal{GL}} = 1$ and $\| h \|_2 \leq r^*(\gamma)$ and $I =  \{ i \in [M]: \| h^*_{G_i}\|_2 \neq 0 \}$.  Let us define $g \in \mathbb R^p$ as
    $$
    g_{G_i} = \left\{
    \begin{array}{ll}
        \frac{h^*_{G_i}}{\| h^*_{G_i} \|_2 } & \mbox{if } i \in I \\
        \frac{h_{G_i}}{\| h_{G_i} \|_2 } & \mbox{if } i \in I^c .
    \end{array}
    \right.
    $$
 We have that  $\| g_{G_i} \|_{2} =1$   for every $i=1,\cdots,M$, and that $\langle g,h^* \rangle = \| h^* \|_{\textnormal{GL}}$. Hence, we have that $g \in \partial (\| \cdot \|_{\textnormal{GL}})_{h^*}$ and it follows that
    \begin{align*}
        \langle h, g \rangle & = \sum_{i \in I} \langle h_{G_i}, g_{G_i} \rangle + \sum_{i \in \notin I}\langle h_{G_i}, g_{G_i} \rangle \geq \sum_{i \in \notin I} \| h_{G_i}\|_2 - \sum_{i \in I} \| h_{G_i}\|_2  \\
        & = \sum_{i=1}^M \| h_{G_i} \|_2 - 2 \sum_{i \in I} \| h_{G_i}\|_2.
    \end{align*}
    Since $\|h\|_{\textnormal{GL}} = \sum_{i=1}^M \| h_{G_i} \|_2 = 1$ and  
    $$
    \sum_{i \in I} \| h_{G_i}\|_2 \leq \sqrt{ | I | } \bigg( \sum_{i \in I}  \| h_{G_i}\|_2^2 \bigg)^{1/2} \leq \sqrt{ | I | } \| h \|_2 \leq \sqrt{ | I | } r^*(\gamma) ,
    $$
    we obtain that $\Delta(\gamma,h^*) \geq 1 - 2\sqrt{| I |} r^*(\gamma) \geq 1/2$ when 
    $| I| \max_{i=1,\cdots,M} | G_i | /n \leq \gamma^2/80$.  
\end{itemize}
Taking $\kappa = 1/\sqrt 3$ we get $\gamma = 1/(144\sqrt 3)$ and applying Theorem~\ref{Main general theorem} concludes the proof of Theorem~\ref{thm groupLasso}.

The proof of Theorem~\ref{Erm groupLasso} follows exactly the same steps. The only difference is the additional upper bound on $\bar \Delta(\gamma,h^*)$. Take
  $$
    g_{G_i} = \left\{
    \begin{array}{ll}
        \frac{h^*_{G_i}}{\| h^*_{G_i} \|_2 } & \mbox{if } i \in I \\
        0 & \mbox{if } i \in I^c .
    \end{array}
    \right.
    $$
 We have that  $\| g_{G_i} \|_{\textnormal{GL}}^* =1$   for every $i=1 \in I$, and that $\langle g,h^* \rangle = \| h^* \|_{\textnormal{GL}}$. Hence, we have that $g \in \partial (\| \cdot \|_{\textnormal{GL}})_{h^*}$. Let $h \in \mathbb H$ such that $\| h  \|_2 = r^*(\gamma)$ and $\| h\| \leq 1$. We have by applying Cauchy-Schwarz twice
    \begin{align*}
        \langle h, g \rangle & = \sum_{i \in I} \langle h_{G_i}, g_{G_i} \rangle \leq \sum_{i \in I} \|h_{G_i}\|_2 \leq \sqrt{|I| \sum_{i \in I} \| h_{G_i} \|_2^2 }\leq \sqrt{|I|} \|h\|_2 = \sqrt{|I|} r^*(\gamma) . 
    \end{align*}

\end{proof}
\subsection{Proofs of Theorem~\ref{corollary matrix} and Theorem~\ref{erm trace regression}}
\begin{proof}
The proof of Theorem~\ref{corollary matrix} is based on Theorem~\ref{Main general theorem} and Lemma~\ref{duality formulation}. We apply the general routine presented in Section~\ref{sec main}. 
\begin{itemize}
    \item[] \textbf{Step 1:} The norm defined on $\mathbb H = \mathbb R^{p_1\times p_2}$ is the nuclear norm $\|\cdot \| =  \| \cdot \|_{S_1}$.\\
    \item[] \textbf{Step 2:} Since $X_1$ is a Gaussian matrix with i.i.d. standard Gaussian entries, we have for any $h \in\mathbb R^{p_1 \times p_2}$ that $\langle h , X_1 \rangle \sim \mathcal N(0,\|h\|_2^2)$ and $\| \langle h , X_1 \rangle \|_{L_4(\mu)}^4 =  3\|h\|_2^4=\|\langle h, X_1 \rangle\|_{L_2(\mu)}^4 $. Hence, by Lemma~\ref{paley argument}, the small-ball assumption is verified for any $\kappa$ in $[0,1)$ and $\delta = (1-\kappa^2)/3$. \\
    \item[] \textbf{Step 3:} By symmetry of $X_i$, the random variable $\varepsilon_i X_i$ is distributed as $X_i$ for any $i \in [n]$ and for any $h \in \mathbb R^{p_1\times p_2}$, $n^{-1/2}\sum_{i=1}^n \varepsilon_i \langle X_i,h \rangle$ is distributed as $\langle h, G \rangle$, where $G$ is a random matrix with i.i.d. standard Gaussian entries. Since the dual norm of the nuclear norm is the operator norm, for any $r>0$, 
    $$
    \mathbb E \sup_{ \substack{\| h \|_{S_1} = 1 \\ \| h\|_2 \leq r}} \langle h , G \rangle \leq \mathbb E \| G \|_{\textnormal{op}} \leq 2 \sqrt {\max(p_1,p_2)} ,
    $$
    where we used Theorem 2.13 in~\cite{DavidsonSzarek01}. It follows that
    $$
    r^*(\gamma) \leq \frac{2}{\gamma} \sqrt \frac{\max(p_1,p_2)}{n} .
    $$
    \item[] \textbf{Step 4:} Since the dual norm of the nuclear norm is the operator norm, Lemma~\ref{duality formulation} yields
    $$
    \| \hat \nu \|_* \leq \frac{\| \xi \|_2}{ \inf_{\nu \in \mathcal S^{n-1}} \| \sum_{i=1}^n \nu_i X_i \|_{\textnormal{op}}} .
    $$
    The following Lemma establishes a lower bound on the infimum above. 
    \begin{lemma} \label{Lemma lower bound spec}
Consider sensing matrices $X_1, \dots, X_n \in \mathbb{R}^{p_1 \times p_2}$ with i.i.d. standard Gaussian entries, $X_{ikl} \overset{i.i.d.}{\thicksim} \mathcal{N}(0,1)$. Suppose that $48n \log(32{n(p_1+p_2)}) \leq p_1p_2$. Then, we have with probability at least $1-2ne^{-p_1p_2/32}$ that 
\begin{align}
    \inf_{v \in \mathcal{S}^{n-1}} \left \| \sum_{i=1}^n v_i X_i \right \|_{op} \geq \frac{ \sqrt{\max(p_1, p_2)}}{2}.
\end{align}
\end{lemma}
The proof of Lemma~\ref{Lemma lower bound spec} is provided in Appendix \ref{supp matrix}.  \\
    \item[] \textbf{Step 5:} We finally turn to the subdifferential condition.  Let $h \in \mathbb R^{p_1 \times p_2}$ such that $\| h \|_{op} = 1$ and $\| h \|_2 \leq r^*(\gamma)$. By Lemma 4.4 in~\cite{lecue2018regularization}, there exists $g \in \partial (\| \cdot \|_{S_1})_{h^*}$ such that $\langle h,g \rangle \geq 1- 4 \sqrt s r^*(\gamma) \geq 1/2$ if  $s \max(p_1,p_2)/n \leq 256\gamma^2$. Hence for $s \max(p_1,p_2)/n \leq 256\gamma^2$, we obtain that $\Delta(\gamma, h^*) \geq 1/2$. 
\end{itemize}
Setting $\kappa = 1/\sqrt 3$, we obtain that $\gamma = 1/(144\sqrt 3)$ and applying Theorem~\ref{Main general theorem} concludes the proof of Theorem~\ref{corollary matrix}.\\

The proof of Theorem~\ref{erm trace regression} follows exactly the same steps. The only difference is the upper bound on $\bar \Delta(\gamma,h^*)$. Since $\text{rank}(h^*) \leq s$, we have $h^* = U \Lambda V^T$, where $\Lambda = \text{diag}(\sigma_1, \cdots, \sigma_s)$, $\sigma_1 \geq \sigma_2 \geq \cdots \geq \sigma_s$ are the singular values of $h^*$ ordered in the non-increasing order and $U \in \mathbb{R}^{p_1 \times s}$ and $V \in \mathbb{R}^{p_2 \times s}$ contain the corresponding left and right singular vectors, respectively. Take $g = U V^T$. We have $\langle g, h^* \rangle = \| h^*\|_{S_1}$ and $\|g\|_{op}=1$ and hence $g \in \partial \left( \|\cdot \| \right)_{h^*}$. 
 For $h \in \mathbb H$ such that $\| h \|_2 = r^*(\gamma)$ and $\| h \|_{S_1} \leq 1$ we have
\begin{align*}
    \langle g , h \rangle = \langle g , U U^T h VV^T \rangle
\end{align*}
by choice of $g$. Observe that $\textnormal{rank}(UU^T h VV^T)\leq \textnormal{rank}(UU^T) = s$. Hence, we obtain that 
$$
 \langle g , h \rangle \leq \|    U U^T h VV^T \|_{S_1} \leq \sqrt{s} \| UU^T h VV^T \|_{2} \leq \sqrt{s} \| h \|_{2}= \sqrt s r^*(\gamma).
$$
\end{proof}

\section*{Acknowledgements}
ML and GC have been funded in part by ETH Foundations of Data Science (ETH-FDS) and gratefully acknowledge helpful discussions with Pedro Teixeira and Afonso Bandeira. 

\appendix

\section{Proof of Lemma~\ref{lemma_interpolation_noise}} \label{lemma supp bp}
\begin{proof}

We denote $\mathbb{X}=(X_1, \dots, X_n)$ and observe that $\sum v_i X_i=\mathbb{X}v=\Sigma^{1/2} \mathbb{Z}v$, where $\mathbb{Z}=(Z_1, \dots, Z_n)$ is a vector consisting of i.i.d. zero mean Gaussian random vectors with covariance matrix $I_p$. By the dual characterization of the $\ell_1$-norm and H\"older's inequality, we have that 
\begin{align*}
    \|\mathbb{Z} v \|_\infty = \sup_{\|b\|_1=1} |\langle b,  \Sigma^{-1/2} \mathbb{X} v \rangle | \leq \|\mathbb{X}v\|_\infty  \sup_{\|b\|_1=1}\| \Sigma^{-1/2} b \|_1.
\end{align*}
Rearranging yields that it suffices to find a uniform lower bound for $\|\mathbb{Z}v\|_\infty$. 
    Let $N_\epsilon$ be an $\epsilon$-net for $\mathcal{S}^{n-1}$. By Lemma 5.2 in \cite{Vershynin12} we have for $\epsilon \leq 1$ that
\begin{align}
    |N_\epsilon| \leq \left ( \frac{3}{\epsilon}\right )^n.
\end{align} \label{bound net card}
 
By definition of $N_\epsilon$, for any $v \in \mathcal{S}^{n-1}$ there exists some $v_{\epsilon} \in N_\epsilon$ such that $\| v_{\epsilon} - v \|_2 \leq \epsilon$. It follows that
$$
\| \mathbb{Z} v \|_\infty \geq  \| \mathbb{Z} v_{\epsilon} \|_\infty - \| \mathbb{Z} (v-v_{\epsilon}) \|_\infty \geq \| \mathbb{Z} v_{\epsilon} \|_\infty - \epsilon \max_{j \in [p]} \| \mathbb{Z}_{\cdot j} \|_{2}.
$$
By a lower tail bound for Gaussians (e.g. Exercise 2.2.8 and 2.23 on p.37 in \cite{GineNickl16}), for a fixed $v_{\epsilon} \in S^{n-1}$ we have that
\begin{align*}
   \mathbb{P} \left (  \| \mathbb{Z} v_{\epsilon} \|_{\infty} \leq t\right ) \leq \bigg( 1-\sqrt{\frac{2}{\pi}}\frac{ e^{-t^2/2}}{t+1} \bigg)^p.
\end{align*}
Choosing $t = \sqrt{\alpha \log(p /n) }$ we obtain that 
\begin{align*}
     \mathbb{P} \left (  \| \mathbb{Z} v_{\epsilon} \|_{\infty} \leq  \sqrt{\alpha \log(p /n) } \right ) & \leq \exp \left ( - \sqrt \frac{2}{\pi}\frac{p ( p/n)^{-\alpha/2} }{\sqrt{\alpha \log( p/n) }+1} \right) \\
     & \leq \exp \left (  - \frac{ n^{\alpha} p^{1-\alpha}}{\sqrt{2\pi} } \right) ,
\end{align*}
where we used the fact that $p\geq e^{1/\alpha} n$ and the inequality $\log(x) \leq x-1$,  $x >0$. By Jensen's inequality we have that $\mathbb{E} \| \mathbb{Z}_{\cdot j } \|_2 \leq \sqrt{n}$ and furthermore, by Borell's inequality  (e.g. Theorem 2.2.7 in \cite{GineNickl16}) and a union bound we have that with probability at least $1-pe^{-n/2}$
\begin{align}
   \max_{j \in [p]} \| \mathbb{Z}_{\cdot j} \|_2 \leq 2\sqrt n \label{bound col X}.
\end{align}
Taking a union bound over $N_\epsilon$, we obtain that with probability at least
$$
1 - \exp \left (  n \log(3/\epsilon)  -  \frac{ n^{\alpha} p^{1-\alpha}}{\sqrt{2\pi} }   \right ) -p \exp(-n/2) ,
$$
the following holds 
$$
\inf_{v \in \mathcal{S}^{n-1}} \| \mathbb{Z} v \|_\infty \geq \sqrt{\alpha \log( p /n) } - 2\epsilon \sqrt n . 
$$
Finally, choosing $\epsilon = \sqrt \alpha /(2 \sqrt{2n})$ and using the fact that $p$ is large enough concludes the proof. 
\end{proof}
\section{Proof of Lemma~\ref{lemma noise interpol group Lasso}} \label{supp group Lasso}
\begin{proof}
As before, we denote $\mathbb{X}=(X_1, \dots, X_n)$ and observe that $\sum v_i X_i = \mathbb{X}v$. We denote $\mathbb{X}_{G_i}:=(X_{kj})_{k \in [n], j \in G_i}$. 
Let $N_\epsilon$ be an $\epsilon$-net for $\mathcal{S}^{n-1}$. By Lemma 5.2 in \cite{Vershynin12} we have again for $\epsilon \leq 1$ that
$   |N_\epsilon| \leq \left ( 3/\epsilon \right)^n.$
By definition of $N_\epsilon$, for any $v \in \mathcal{S}^{n-1}$ there exists some $v_{\epsilon} \in N_\epsilon$ such that $\| v_{\epsilon} - v \|_2 \leq \epsilon$. It follows that
\begin{align*}
    \max_{i \in [M]} \| \mathbb{X}_{G_i} v \|_2 & \geq  \max_{i \in [M]} \| \mathbb{X}_{G_i} v_{\epsilon} \|_2  - \max_{i \in [M]} \| \mathbb{X}_{G_i} (v-v_{\epsilon}) \|_2 \\
    & \geq \max_{i \in [M]} \| \mathbb{X}_{G_i} v_{\epsilon} \|_2 - \epsilon \max_{i \in [M]} \| \mathbb X_{ G_i} \|_{\textnormal{op}}.
\end{align*}
The expression $\| \mathbb X_{G_i} v_{\epsilon} \|_2^2$ is $\chi^2_{|G_i|}$ distributed with $|G_i|$ degrees of freedom. Hence, by lower tail bounds for $\chi^2$-variables (e.g. Lemma 1 in \cite{LaurentMassart00}), for  fixed $v_{\epsilon} \in S^{n-1}$ we have that
\begin{align*}
  &  \mathbb{P} \left ( \max_{i \in [M]} \left [\| \mathbb X_{G_i} v_{\epsilon} \|_2^2 - |G_i|+2\sqrt{|G_i| t_i} \right ] \leq 0 \right ) \\ =&    \prod_{i \in [M] }  \mathbb{P} \left (  \| \mathbb X_{G_i} v_{\epsilon} \|_2^2 \leq  |G_i|-2\sqrt{|G_i| t_i} \leq 0 \right ) \leq e^{-\sum_{i \in [M]} t_i }.
\end{align*}
Choosing $t_i=|G_i|/16$, $i \in [M]$, we obtain that 
\begin{align*}
   \mathbb{P} \left ( \max_{i \in [M]} \| \mathbb X_{G_i} v_{\epsilon} \|_2^2 \leq \min_{i \in [M]} |G_i|/\sqrt{2} \right ) \leq \exp \left ( \sum_i |G_i|/16\right ) = e^{-p/16}.
\end{align*}
The matrix $\mathbb X_{G_i}$ is a matrix with i.i.d. Gaussian entries and of size $ |G_i| \times n $. Hence, by Theorem 5.35 in~\cite{Vershynin12}, with probability at least $1-\exp(-t/2)$
$$
\| \mathbb X_{ G_i} \|_{\textnormal{op}} \leq \sqrt n + \sqrt{| G_i |} + \sqrt t 
$$
By using an union bound and choosing $t = n$, we obtain that with probability at least $1-e^{-n/2 + \log(M)}$
$$
\| \mathbb X_{ G_i} \|_{\textnormal{op}} \leq 2\sqrt n + \sqrt{\max_{i \in [M]} | G_i |} .
$$
Taking an union bound over $N_\epsilon$, we obtain that with probability at least
$$
1 - e^{n\log(3/\epsilon)-p/16} - e^{-n/2 + \log(M)} ,
$$
the following holds 
$$
\inf_{v \in \mathcal{S}^{n-1}}  \max_{i \in [M]} \| \mathbb X_{G_i} v \|_2 \geq \sqrt{\min_{i \in [M]} |G_i|/2} - \epsilon \bigg(  2\sqrt n + \sqrt{\max_{i \in a [M]} | G_i|} \bigg). 
$$
Finally, choosing $\epsilon = \sqrt{\min_{i \in [M]} |G_i|}/(2\sqrt 2( 2\sqrt n + \sqrt{\max_{i \in [M]} |G_i|}))$ and using the fact that
$$
p \geq 32 \log \bigg(6\sqrt 2 ( 2\sqrt n + \sqrt W )\bigg) ,$$
where $W = \max_{i \in [M] } |G_i| / \min_{i \in [M]} |G_i|$ concludes the proof. 
\end{proof}
\section{Proof of Lemma~\ref{Lemma lower bound spec}} \label{supp matrix}
\begin{proof}
Without loss of generality, we assume that $p_1 \geq p_2$. For fixed $v \in \mathcal{S}^{n-1}$ we have that $\sum_{i=1}^n v_i X_i \overset{d}{=} G$ for some matrix $G \in \mathbb{R}^{p_1 \times p_2}$ consisting of i.i.d. standard Gaussian entries. Since $\| \cdot \|_{op} \geq \max_{j \in [p_2]} \| \cdot e_j\|_2$ and the columns of $G$ are independent and identically distributed, we have that
\begin{align}
    \mathbb{P} \left ( \|G\|_{op}^2 \leq t\right ) & \leq \mathbb{P} \left ( \max_{j \in [p_2]} \|Ge_j\|_2^2 \leq t \right )  =\left ( \mathbb{P} \left ( \|Ge_1\|_2^2 \leq t \right ) \right )^{p_2}.
\end{align}
It is thus left to find a sufficient lower tail bound for $\|Ge_1\|_2^2$. $\|Ge_1\|_2^2$ is $\chi_{p_1}^2$-distributed. Hence, we have  that $\mathbb{E} \|Ge_1\|_2^2 = p_1$. Moreover, by lower tail bounds for $\chi_{p_1}^2$ variables, e.g. Lemma 1 in \cite{LaurentMassart00}, we have that
\begin{align}
\mathbb{P} \left ( \|Ge_j\|_2^2 \leq p_1-2\sqrt{p_1t} \right ) \leq e^{-t}.
\end{align}
Hence, choosing $t=\frac{p_1}{16}$, we obtain that
\begin{align}
    \mathbb{P} \left ( \|G\|_{op} \leq \sqrt{\frac{{p_1}}{2}}\right ) \leq e^{-p_1p_2/16}.
\end{align}
Now let $N_\epsilon$ be an $\epsilon$-net of $\mathcal{S}^{n-1}$. We recall that by Lemma 5.2. in \cite{Vershynin12} $|N_\epsilon|\leq (3/\epsilon)^n$ for $\epsilon \leq 1$. By definition of $N_\epsilon$, for every $v \in \mathcal{S}^{n-1}$ there exists  $v_\epsilon \in N_\epsilon$ such that
\begin{align}
   \left \| \sum v_i X_i \right \|_{op} & \geq \left \| \sum (v_{\epsilon})_i X_i \right \|_{op} - \|v-v_{\epsilon}\|_1 \max_{i=1, \dots, n} \|X_i\|_{op} \notag \\ 
   & \geq \left \| \sum (v_{\epsilon})_i X_i \right \|_{op} - \sqrt{n}\epsilon\max_{i=1, \dots, n} \|X_i\|_{op}.
\end{align}
We have by Theorem 2.13 in \cite{DavidsonSzarek01} that $\mathbb{E}\|X_i\|_{op}\leq \sqrt{p_1}+\sqrt{p_2}\leq 2 \sqrt{\max(p_1,p_2)}$. Moreover, by Borell's inequality (Theorem 2.2.7 in \cite{GineNickl16}) and a union bound, we have with probability at least $1-ne^{-p_1p_2}$ that
\begin{align}
    \max_{i \in [n]} \|X_i\|_{op} \leq 2 \sqrt{\max(p_1,p_2)} + 2\sqrt{p_1p_2} \leq 4  \sqrt{p_1p_2}.
\end{align}
Hence, choosing $\epsilon=\frac{1}{10\sqrt{np_2}}$ and using an union bound we obtain that 
\begin{align}
   \mathbb{P} \left (  \inf_{v \in \mathcal{S}^{n-1}} \left \| \sum_{i=1}^n v_i X_i\right \| < \frac{\sqrt{p_1}}{2} \right ) & \leq ne^{-p_1p_2}+|N_\epsilon| \mathbb{P} \left ( \|G\|_{op} \notag  \leq \sqrt{\frac{{p_1}}{2}}\right ) \\ & \leq 2 ne^{-p_1 p_2 /32}, \notag 
\end{align}
provided that $3n \log(16\sqrt{np_2}) \leq p_1p_2/32$.
\end{proof}
\section{Proof of Proposition \ref{Prop lower bd gen}} \label{appendix proof lb}
\begin{proof}
By assumption there exists some $h_1 \in L$ such that $\|\langle h_1, X \rangle \|_{L_2(\mu)}=\epsilon^2/8$. 
We define two distributions: Under $\mathbb{P}_0$ data is generated as
$$Y_i=\langle X_i, h_1 \rangle,~~~~i=1, \dots, n,$$
whereas under $\mathbb{P}_1$
$$Y_i=\langle X_i, h_0\rangle+\xi_i, ~~~~h_0=0, ~\xi_i=\langle X_i, h_1 \rangle, ~~~~i=1, \dots, n.$$
We see that $\mathbb{P}_0=\mathbb{P}_1$. Moreover, $\|\langle X,h_0-h_1\rangle\|_{L_2(\mu)}^2=\|\langle h_1, X \rangle \|_{L_2(\mu)}^2=\epsilon^2/8$. 
Since the $X_i$ are i.i.d. we have that 
\begin{align*}
    \mathbb{E}\|\xi\|^2_2 =n \mathbb{E} |\langle X_1, h_1 \rangle|^2 = n \|\langle X, h_1 \rangle \|_{L_2(\mu)}^2=n\epsilon^2/8. 
\end{align*}
Hence, by Markov's inequality, with probability at least $7/8$
\begin{align*}
    \|\xi\|^2_2 \leq n \epsilon^2. 
\end{align*}
Finally, accounting for the event where $\|\xi\|^2_2 > n \epsilon^2$ and afterwards essentially applying Le Cam's two point Lemma we obtain
\begin{align*}
   &   \inf_{\tilde h} \sup_{h^* \in L,  ~\xi: \|\xi\|_2^2\leq n \epsilon^2 } \mathbb{P}_{h^*, \xi} \left(   \|\langle \tilde h-h^*, X \rangle\|_{L_2(\mu)}^2  \geq  \frac{ \epsilon^2}{16} \right)\\   \geq  &  \inf_{\tilde h} \max_{k \in \{0,1\}}   \mathbb{P}_k \left ( \|\langle \tilde h-h_k, X \rangle \|_{L_2(\mu)}^2 \geq \frac{\epsilon^2}{16} \right )-\frac{1}{8} \\
       = & \inf_{\tilde h} \max \left ( \mathbb{P}_0 \left ( \|\langle\tilde h, X \rangle\|_{L_2(\mu)}^2 \geq \frac{\epsilon^2}{16} \right ), \mathbb{P}_0 \left ( \|\langle \tilde h-h_1, X \rangle\|_{L_2(\mu)}^2  \geq \frac{\epsilon^2}{16} \right )\right ) -\frac{1}{8}  \geq \frac{3}{8} ,
\end{align*}
where the last inequality follows by noting that $$\mathbb{P}_0 \left ( \{ \|\langle \tilde h, X \rangle\|_{L_2(\mu)}^2 \geq \frac{\epsilon^2}{16} \} \cup  \{ \|\langle \tilde h-h_1, X \rangle\|_{L_2(\mu)}^2 \geq \frac{\epsilon^2}{16} \} \right ) =1$$ as $\|\langle h_1, X \rangle \|_{L_2(\mu)}^2\geq \epsilon^2/8$.

\end{proof}

\bibliography{thesisBibliography}
\bibliographystyle{alpha}

\end{document}